\documentclass[reqno,11pt,a4paper,final]{amsart}
\usepackage[a4paper,left=35mm,right=35mm,top=35mm,bottom=35mm,marginpar=25mm]{geometry} 
\usepackage{amsmath}
\usepackage{amssymb}
\usepackage{amsthm}
\usepackage{amscd}
\usepackage{esint}
\usepackage[ansinew]{inputenc}
\usepackage{cite}
\usepackage{bbm}
\usepackage{xcolor}
\usepackage[english=american]{csquotes}
\usepackage[final]{graphicx}
\usepackage{hyperref}
\usepackage{calc}
\usepackage{bm}
\usepackage{enumerate}
\usepackage{transparent}

\graphicspath{{Pictures/}}

\numberwithin{equation}{section}

\newtheoremstyle{thmlemcorr}{10pt}{10pt}{\itshape}{}{\bfseries}{.}{10pt}{{\thmname{#1}\thmnumber{ #2}\thmnote{ (#3)}}}
\newtheoremstyle{thmlemcorr*}{10pt}{10pt}{\itshape}{}{\bfseries}{.}\newline{{\thmname{#1}\thmnumber{ #2}\thmnote{ (#3)}}}
\newtheoremstyle{defi}{10pt}{10pt}{\itshape}{}{\bfseries}{.}{10pt}{{\thmname{#1}\thmnumber{ #2}\thmnote{ (#3)}}}
\newtheoremstyle{remexample}{10pt}{10pt}{}{}{\bfseries}{.}{10pt}{{\thmname{#1}\thmnumber{ #2}\thmnote{ (#3)}}}
\newtheoremstyle{ass}{10pt}{10pt}{}{}{\bfseries}{.}{10pt}{{\thmname{#1}\thmnumber{ A#2}\thmnote{ (#3)}}}

\theoremstyle{thmlemcorr}
\newtheorem{theorem}{Theorem}
\numberwithin{theorem}{section}
\newtheorem{lemma}[theorem]{Lemma}
\newtheorem{corollary}[theorem]{Corollary}
\newtheorem{proposition}[theorem]{Proposition}

\newtheorem{conjecture}[theorem]{Conjecture}

\theoremstyle{thmlemcorr*}
\newtheorem{theorem*}{Theorem}
\newtheorem{lemma*}[theorem]{Lemma}
\newtheorem{corollary*}[theorem]{Corollary}
\newtheorem{proposition*}[theorem]{Proposition}
\newtheorem{problem*}[theorem]{Problem}
\newtheorem{conjecture*}[theorem]{Conjecture}

\theoremstyle{defi}

\theoremstyle{remexample}
\newtheorem{remark}[theorem]{Remark}
\newtheorem{example}[theorem]{Example}

\theoremstyle{ass}

\newcommand{\Crm}{\mathrm{C}}

\newcommand{\Lrm}{\mathrm{L}}

\newcommand{\Wrm}{\mathrm{W}}

\newcommand{\Acal}{\mathcal{A}}

\newcommand{\Ecal}{\mathcal{E}}
\newcommand{\Fcal}{\mathcal{F}}

\newcommand{\Hcal}{\mathcal{H}}
\newcommand{\Ical}{\mathcal{I}}

\newcommand{\Lcal}{\mathcal{L}}
\newcommand{\Mcal}{\mathcal{M}}

\newcommand{\Pcal}{\mathcal{P}}

\newcommand{\Ucal}{\mathcal{U}}

\newcommand{\Ebf}{\mathbf{E}}

\newcommand{\Mbf}{\mathbf{M}}

\newcommand{\Pbf}{\mathbf{P}}
\newcommand{\Qbf}{\mathbf{Q}}

\newcommand{\Ybf}{\mathbf{Y}}

\newcommand{\Abb}{\mathbb{A}}
\renewcommand{\Bbb}{\mathbb{B}}

\newcommand{\Sbb}{\mathbb{S}}

\DeclareMathOperator{\id}{id}

\DeclareMathOperator{\supmod}{sup}
 
\DeclareMathOperator{\curl}{curl}
\DeclareMathOperator{\dist}{dist}
\DeclareMathOperator{\rank}{rank}

\DeclareMathOperator{\spn}{span}
\DeclareMathOperator{\ran}{ran}
\DeclareMathOperator{\supp}{supp}
\newcommand{\ee}{\mathrm{e}}
\newcommand{\ii}{\mathrm{i}}

\newcommand{\setb}[2]{\bigl\{\, #1 \ \ \textup{\textbf{:}}\ \ #2 \,\bigr\}}

\newcommand{\setBB}[2]{\biggl\{\, #1 \ \ \textup{\textbf{:}}\ \ #2 \,\biggr\}}

\newcommand{\norm}[1]{\|#1\|}

\newcommand{\abs}[1]{|#1|}

\newcommand{\absn}[1]{|#1|}
\newcommand{\absb}[1]{\bigl|#1\bigr|}

\newcommand{\absBB}[1]{\biggl|#1\biggr|}

\newcommand{\floor}[1]{\lfloor #1 \rfloor}

\newcommand{\dpr}[1]{\langle #1 \rangle}	

\newcommand{\dprn}[1]{\langle #1 \rangle}
\newcommand{\dprb}[1]{\bigl\langle #1 \bigr\rangle}

\newcommand{\ddprb}[1]{\bigl\langle\hspace{-2.5pt}\bigl\langle #1 \bigr\rangle\hspace{-2.5pt}\bigr\rangle}

\newcommand{\cl}[1]{\overline{#1}}
\newcommand{\di}{\mathrm{d}}
\newcommand{\dd}{\;\mathrm{d}}

\newcommand{\N}{\mathbb{N}}
\newcommand{\R}{\mathbb{R}}
\newcommand{\C}{\mathbb{C}}

\newcommand{\loc}{\mathrm{loc}}
\newcommand{\sym}{\mathrm{sym}}
\newcommand{\skw}{\mathrm{skew}}

\newcommand{\ONE}{\mathbbm{1}}

\newcommand{\toweakstar}{\overset{*}\rightharpoonup}

\newcommand{\todown}{\downarrow}

\newcommand{\embed}{\hookrightarrow}

\newcommand{\conv}{\star}

\newcommand{\sbullet}{\begin{picture}(1,1)(-0.5,-2)\circle*{2}\end{picture}}
\newcommand{\frarg}{\,\sbullet\,}
\newcommand{\BV}{\mathrm{BV}}
\newcommand{\BD}{\mathrm{BD}}
\newcommand{\LD}{\mathrm{LD}}

\newcommand{\BDY}{\mathbf{BDY}}

\newcommand{\toY}{\overset{\Ybf}{\to}}

\newcommand{\eps}{\epsilon}

\DeclareMathOperator{\Curl}{Curl}
\DeclareMathOperator{\Tan}{Tan}
\DeclareMathOperator{\Div}{div}

\DeclareMathOperator{\Gr}{Gr}

\newcommand{\term}[1]{\emph{#1}}

\newcommand{\proofstep}[1]{\textit{#1}}

\newcommand{\Rdds}{\R^{d \times d}_\sym}

\def\Xint#1{\mathchoice 
{\XXint\displaystyle\textstyle{#1}}%
{\XXint\textstyle\scriptstyle{#1}}%
{\XXint\scriptstyle\scriptscriptstyle{#1}}%
{\XXint\scriptscriptstyle\scriptscriptstyle{#1}}%
\!\int} 
\def\XXint#1#2#3{{\setbox0=\hbox{$#1{#2#3}{\int}$} 
\vcenter{\hbox{$#2#3$}}\kern-.5\wd0}} 
\def\dashint{\,\Xint-}

\newcommand{\restrict}{\begin{picture}(10,8)\put(2,0){\line(0,1){7}}\put(1.8,0){\line(1,0){7}}\end{picture}}

\renewcommand{\eps}{\varepsilon}
\renewcommand{\phi}{\varphi}

\title[Fine properties of BD maps]{Fine properties of functions of bounded deformation -- an approach via linear PDEs}

\author[G.~De Philippis]{Guido De Philippis}
\address{\textit{G.~De Philippis:} Courant Institute of Mathematical Sciences, New York University, 251 Mercer St., New York, NY 10012, USA.}
\email{guido@cims.nyu.edu}

\author[F.~Rindler]{Filip Rindler}
\address{\textit{F.~Rindler:} Mathematics Institute, University of Warwick, Coventry CV4 7AL, UK, and The Alan Turing Institute, British Library, 96 Euston Road, London NW1 2DB, UK.}
\email{F.Rindler@warwick.ac.uk}

\hypersetup{
  pdfauthor = {Guido De Philippis (Courant Institute of Mathematics) and Filip Rindler (University of Warwick)},
  pdftitle = {Fine properties of functions of bounded deformation -- an approach via linear PDEs},
  pdfsubject = {},
  pdfkeywords = {}
}

\begin{document}

\maketitle


\begin{abstract}
In this survey we collect some recent results obtained by the authors and collaborators concerning the fine structure of functions of bounded deformation (BD). These maps are $\mathrm{L}^1$-functions with the property that the symmetric part of their distributional derivative is representable as a bounded (matrix-valued) Radon measure. It has been known for a long time that for a (matrix-valued) Radon measure the property of being a symmetrized gradient can be characterized by an under-determined second-order PDE system, the Saint-Venant compatibility conditions. This observation gives rise to a new approach to the fine properties of BD-maps via the theory of PDEs for measures, which complements and partially replaces classical arguments. Starting from elementary observations, here we elucidate the ellipticity arguments underlying this recent progress and give an overview of the state of the art. We also present some open problems.

%
%

\vspace{4pt}

\noindent\textsc{Date:} \today{}. 
\end{abstract}


\section{Introduction}

In this survey we review a PDE approach to the study of fine properties of functions of bounded deformation (BD), which was recently developed by the authors. In particular, we will show how this approach allows to characterize the structure of the singular part of the symmetrized derivative and also to recover some known structure properties of these functions in an easier and more robust way.

Functions of bounded deformation are fundamental in the analysis of a large number of problems from mechanics, most notably in the theory of (linearized) elasto-plasticity, damage, and fracture; we refer to~\cite{Temam83book,FuchsSeregin00book,AmbrosioCosciaDalMaso97,Suquet78,Suquet79,MatthiesStrangChristiansen79,Kohn79thesis,TemamStrang80,Kohn82} and the references contained therein. A common feature of these theories is that the natural coercivity of the problem only yields a-priori bounds in the $\Lrm^1$-norm of the symmetrized gradient 
\begin{equation*} \label{eq:sym_grad}
  \Ecal u := \frac{1}{2} \bigl( \nabla u + \nabla  u^T \bigr)
\end{equation*}
of a map $u \colon \Omega \subset \R^d \to \R^d$. As $\Lrm^1$ is not reflexive, such $\Lrm^1$-norm bounds do not allow for the selection of weakly converging subsequences. This issue is remedied by enlarging $\Lrm^1$ to the space of finite Radon measures, where now a uniform bound on the total variation norm permits one to select a weakly* converging subsequence.

Given an open set \(\Omega\subset \R^d\) with Lipschitz boundary, the space $\BD(\Omega)$ of \term{functions of bounded deformation} is then the space of functions $u \in \Lrm^1(\Omega;\R^d)$  such that the distributional \term{symmetrized derivative} 
\[
  Eu := \frac{1}{2} \bigl( Du + Du^T \bigr)
\]
is (representable as) a finite Radon measure, $Eu \in \Mcal(\Omega;\Rdds)$, where \(\Rdds\) denotes the space of \(d\times d\) symmetric matrices. The space $\BD(\Omega)$ of functions of bounded deformation is a (non-reflexive) Banach space under the norm
\[
  \norm{u}_{\BD(\Omega)} := \norm{u}_{\Lrm^1(\Omega;\R^d)} + \abs{Eu}(\Omega),
\]
where $\abs{Eu}$ denotes the total variation measure of $Eu$.

One particularly important feature of BD-maps is that the symmetrized derivative $Eu$ may contain a \emph{singular part}, i.e., a measure that is not absolutely continuous with respect to Lebesgue measure. This may, for instance, correspond to concentrations of strain in the model under investigation. As we will show in the sequel,  the allowed ``shapes'' of these concentrations are quite restricted, merely due to the fact that they occur in a symmetrized gradient. These \emph{rigidity} considerations  play a prominent role in the analysis of a model, for instance, in the integral representation and lower semicontinuity theory for functionals defined on \(\BD\) and in the characterization of Young measures generated by symmetrized gradients.

\subsection*{Fine structure results}

The study of the fine structure of BD-maps started in the PhD thesis of Kohn~\cite{Kohn79thesis}, and was then systematically carried out by Ambrosio--Coscia--Dal Maso in~\cite{AmbrosioCosciaDalMaso97}; further recent results on the fine properties of $\BD$ can be found in~\cite{Babadjian15,ContiFocardiIurlano18,GmeinederRaita18?,SpectorVanSchaftingen18?} and the references therein. Classically, the analysis of the space $\BD(\Omega)$ has been modelled on the analysis of the space of \term{functions of bounded variation}, \(\BV(\Omega;\R^\ell)\),  i.e., those functions \(u\in \Lrm^1(\Omega;\R^\ell)\) such that the distributional derivative \(Du\) can be represented as a finite Borel measure, \(Du \in \Mcal(\Omega;\R^{\ell\times d})\), where \(\R^{\ell\times d}\) is the space of \(\ell \times d\) matrices. The theory of BV-maps in multiple dimensions goes back to De Giorgi~\cite{DeGiorgi54} and has become a fundamental tool in the Calculus of Variations and in Geometric Measure Theory.

Starting from the seminal works of Federer,~\cite{Federer69book}, we now have a complete understanding of  the fine structure of BV-maps, which we summarize in the following, see~\cite{AmbrosioFuscoPallara00book} for details and proofs. For  a vector  Radon measure \(\mu\in \Mcal(\R^d;\R^n)\) we write \(\mu=\mu^a+\mu^s\) for its Lebesgue--Radon--Nikod\'{y}m decomposition with respect to Lebesgue measure. Denoting by \(\abs{\mu}\) the total variation measure of $\mu$, we call
 \[
 \frac{\di \mu}{\di \abs{\mu}}(x):=\lim_{r\to 0}  \frac{\mu(B_r(x))}{\abs{\mu}(B_r(x))}
 \] 
the \emph{polar vector}, whose existence for \(\abs{\mu}\)-almost every $x$ is ensured by the Besicovitch differentiation theorem (see~\cite[Theorem~2.22]{AmbrosioFuscoPallara00book}). We then have the following (amalgamated) structure result in $\BV$:

\begin{theorem}\label{thm:Bvstruct} For \(u\in \BV_\loc(\R^d;\R^\ell)\), we  can decompose \(Du\) as
\[
Du=D^a u+D^s u=D^a u+D^ju+D^c u.
\]
Here:
\begin{itemize}
\item[(i)] \(D^a u=\nabla u \, \Lcal^d\)  is the \term{absolutely continuous part} of \(Du\) (with respect to Lebesgue measure). Its density \(\nabla u\in \Lrm^1_\loc(\Omega;\R^{\ell\times d})\) is the \term{approximate gradient} of \(u\), which satisfies for \(\Lcal^d\)-almost every \(x\) that
\[
\lim_{r\to 0} \dashint_{B_{r}(x)} \frac{\abs{u(y)-u(x)-\nabla  u(x)(y-x)}}{\abs{y-x}} \dd y =0.
\] 

\item[(ii)] \(D^j u\) is  the \term{jump part} of \(Du\). It is  concentrated on a countably rectifiable \(\Hcal^{d-1}\) \(\sigma\)-finite set \(J_u\), where it can be represented as 
\[
D^j u= (u^+-u^-)\otimes  \nu_{J_u} \, \Hcal^{d-1}\restrict J_u.
\]
Here, \(\nu_{J_u}\) is normal to \(J_u\) and \(u^\pm\) are the traces of \(u\) on \(J_u\) in positive and negative \(\nu_{J_u}\)-direction, respectively (note that the product \((u^+-u^-)\otimes  \nu_{J_u}\) does not depend on the choice of the orientation), and \((a\otimes b)_{ij}:=a_ib_j\) is the tensor product of the vectors $a$ and $b$.

\item[(iii)] \(D^c u\) is the \term{Cantor part} of \(Du\). It vanishes on every \(\Hcal^{d-1}\) \(\sigma\)-finite set. Furtheremore, for  \(|D^c u|\)-almost all \(x\) there are \(a(x)\in \R^\ell\) and \(b(x)\in \R^d\) such that 
\[
\frac{\di D u}{\di |D u|} (x)=a(x)\otimes b(x).
\]
\end{itemize}
In particular,  \(|Du|\ll \Hcal^{d-1}\), that is, $|Du|$ is absolutely continuous with respect to the $(d-1)$-dimensional Hausdorff  measure $\Hcal^{d-1}$.  Furthermore, if one denotes by \(C_{u}\) the set of \term{approximate  continuity points} of \(u\), i.e., those \(x\)  for which there exists a \(\lambda(x) \in \R^\ell\) such that 
\[
\lim_{r \to 0} \dashint_{B_{r}(x)}\abs{u(y)-\lambda(x)}\dd y=0,
\]
and by \(S_u:=\R^d \setminus C_u\) the set of \term{approximate  discontinuity points}, then 
\begin{equation}\label{e:discontinuity}
\Hcal^{d-1}(S_u\setminus J_u)=0  \qquad\text{and}\qquad |D^s u|(S_u \setminus J_u)=0.
\end{equation}
\end{theorem}

The above structural results are fundamental in the study of many variational problems involving functions of bounded variation. In particular,~(iii) above is known as \emph{Alberti's rank-one theorem},  a key structural result for BV-maps first proved in~\cite{Alberti93} (see also~\cite{MassaccesiVittone19} for a recent, more streamlined proof). It entails strong constraints on  the type of possible singularities for \(Du\), see Corollary~\ref{cor:Alberti_1D} below.

The proofs of all these properties of BV-maps  rely heavily on the connection between functions of bounded variation and \emph{sets of finite perimeter} and on the fine properties of such sets~\cite{AmbrosioFuscoPallara00book,Maggi12book}. This link is expressed by the Fleming--Rishel coarea formula~\cite{FlemingRishel60}: For all \(u \colon \R^d \to \R\) it holds that 
\[
  Du = \int_{-\infty}^{+\infty} D \ONE_{\{ u>t\}} \dd t,\qquad  \abs{Du} = \int_{-\infty}^{+\infty} \abs{D \ONE_{\{ u>t\}}} \dd t,
\]
where both equalities are to be understood in the sense of measures.

Clearly, \(\BV \subset \BD\), but it has been known since the work of  Ornstein~\cite{Ornstein62} that the inclusion is strict (however, see~\cite{Friedrich17,Friedrich18,ContiFocardiIurlano18,ChambolleContiIurlano19} for interesting recent results on partial converses under additional assumptions). More precisely, one can show that for all $N \in \N$ there exists a map $u \in \Wrm^{1,\infty}_0(B_1;\R^d)$ (where $B_1 := B_1(0)$ is the unit ball in $\R^d$) such that
\[
  \inf_{F \in \R^{d \times d}} \int_{B_1} \abs{D u - F} \dd x
  \geq N \int_{B_1} \absBB{\frac{D u + Du^T}{2}} \dd x,
\]
see~\cite{Ornstein62,ContiFaracoMaggi05,KirchheimKristensen16} and also~\cite[Theorem~9.26]{Rindler18book}. This implies that a Korn inequality of the form
\begin{equation}\label{e:korn}
  \| \nabla u\|_{\Lrm^p}^p \lesssim \|u\|_{\Lrm^p}^p + \| \Ecal u\|_{\Lrm^p}^p ,  \qquad u \in \Crm^\infty(B_1;\R^d),  
\end{equation}
fails for \(p=1\) (while it is true for all \(p\in (1,+\infty)\), see~\cite{Temam83book}). Furthermore, no  analogue of the coarea formula is known in $\BD$ and this prevents the application of several techniques used to establish Theorem~\ref{thm:Bvstruct}. Nevertheless, several results  have been obtained and the analogues of  the first two points of Theorem~\ref{thm:Bvstruct} have been known for many years~\cite{AmbrosioCosciaDalMaso97,Kohn79thesis}:

\begin{theorem}\label{thm:BdstructIntro}
For \(u\in \BD_\loc(\R^d)\) we  can decompose \(Eu\) as
\[
Eu=E^a u+E^s u=E^a u+E^ju+E^c u.
\]
Here:
\begin{itemize}
\item[(i)] \(E^a u=\Ecal  u  \, \Lcal^d\)  is the  \emph{absolutely continuous part} of \(Eu\). Its density \(\Ecal u\in \Lrm^1_\loc(\Omega;\Rdds)\) is the \term{approximate  symmetrized gradient} of \(u\), which satisfies for \(\Lcal^d\)-almost every \(x\) that
\[
\lim_{r\to 0} \dashint_{B_{r}(x)} \frac{\abs{(u(y)-u(x))\cdot(y-x)-(\Ecal u(x)(y-x))\cdot (y-x))}}{|y-x|^2}=0.
\] 
\item[(ii)] \(E^j u\) is  the \emph{jump part} of \(Eu\). It is  concentrated on a countably rectifiable \(\Hcal^{d-1}\) \(\sigma\)-finite set \(J_u\), where it can be represented as 
\[
E^j u= (u^+-u^-)\odot  \nu_{J_u} \, \Hcal^{d-1}\restrict J_u.
\]
Here, \(\nu_{J_u}\) is normal to \(J_u\) and \(u^\pm\) are the traces of \(u\) on \(J_u\) in positive and negative \(\nu_{J_u}\)-direction, respectively, and \((a\odot b)_{ij}:=\frac{1}{2}(a \otimes b + b \otimes a)\) is the symmetric tensor product of the vectors $a,b$.

\item[(iii)] \(E^c u\) is the \emph{Cantor part} of \(Eu\). It vanishes on every \(\Hcal^{d-1}\) \(\sigma\)-finite set.
\end{itemize}
In particular, \(|E u|\ll \Hcal^{d-1}\).
\end{theorem}

Concerning the trace, we remark that there exist two bounded linear trace operators onto $\Hcal^{d-1}$-rectifiable sets, giving the one-sided traces $u^\pm$, see Theorem~II.2.1 of~\cite{TemamStrang80} and also~\cite{Babadjian15,BreitDieningGmeineder17?}.

Despite the clear similarity between Theorem~\ref{thm:Bvstruct} and Theorem~\ref{thm:BdstructIntro}, some parts are missing. The analogue of the first statement in~\eqref{e:discontinuity} is currently unknown and only partial results are available. This is one of the major open problems in the theory of BD-maps:

\begin{conjecture}
For all \(u\in \BD_\loc(\R^d)\) it holds that $\Hcal^{d-1}(S_u\setminus J_u)=0$.
\end{conjecture}

We remark that the following weaker statement was proved in~\cite{AmbrosioCosciaDalMaso97}: If $u \in \BD_\loc(\R^d)$, then $|Ev|(S_u \setminus   J_u)=0$ for all $v \in \BD_\loc(\R^d)$; in particular, $|E^s u|(S_u \setminus J_u)=0$.

On the other hand, the analogue of Alberti's rank-one theorem has recently been established by the authors~\cite{DePhilippisRindler16}:

\begin{theorem}\label{thm:AlbertiBD} Let \(u\in \BD_\loc(\R^d)\). Then, for \(|E^s u|\)-almost every \(x\) there are vectors \(a(x), b(x) \in \R^d\) such that 
\[
\frac{\di E u}{\di |Eu|} (x)=a(x)\odot b(x). 
\]
\end{theorem}

Note that for the jump part the above theorem is already contained in Theorem~\ref{thm:BdstructIntro}~(ii) and the real difficulty lies in dealing with the Cantor part of \(Eu\). Like Alberti's rank one theorem, Theorem~\ref{thm:AlbertiBD} allows to deduce some quite precise information on the structure of the singularities of \(Eu\), see Section~\ref{s:singloc} below. The picture is however still  less complete than in the BV-case, see Conjecture~\ref{conj:strcuture}.

\subsection*{A new approach to study singularities}

The failure of a coarea-type formula makes the approach used in~\cite{AmbrosioCosciaDalMaso97} unsuitable for the proof of Theorem~\ref{thm:AlbertiBD}. The strategy followed in~\cite{DePhilippisRindler16} is instead based on a new point of view combining Harmonic Analysis techniques with some tools from Geometric Measure Theory. This approach is heavily inspired by the ideas of Murat and Tartar in the study of compensated compactness~\cite{Tartar79,Murat78,Murat79,Tartar83} and has been introduced in this context for the first time in the PhD thesis of the second author~\cite{Rindler11thesis}. 

The core idea is  to ``forget'' about the map $u$ itself and to work with $Eu$ only. This  is enabled by the fact that symmetrized derivatives are not arbitrary measures (with values in $\Rdds$), but that they satisfy a PDE constraint, namely the \term{Saint-Venant compatibility conditions}: If the measure $\mu=(\mu_{jk})$ is the symmetrized derivative of some $u\in \BD(\Omega)$, i.e., $\mu=Eu $, then, by direct computation,
\[
  \sum_{i=1}^d \partial_{ik} \mu_{ij}+\partial_{ij} \mu_{ik}-\partial_{jk} \mu_{ii}-\partial_{ii} \mu_{jk} = 0  \qquad
  \text{for all $j,k=1,\ldots,d$.}
\]
For $d=3$ this constraint can be written as the vanishing of a  double application of the matrix-curl, defined as the matrix-valued differential operator 
\begin{equation*}\label{e:Curl}
(\Curl A)_{ij} := \sum_{k,l=1}^3 \eps_{ilk} \partial_l A_{jk},\qquad i,j\in\{1,2,3\},
\end{equation*}
 where $\eps_{ilk}$ denotes the parity of the permutation $\{1,2,3\}\to\{i,l,k\}$. Hence, we will write the above equations (for all dimensions) in shortened form as
\[
  \Curl \Curl \mu = 0
\]
and say that $\mu$ is ``$\Curl\Curl$-free''. This PDE-constraint furthermore  contains \emph{all} the information about symmetrized derivatives, as $\Curl\Curl$-freeness is both necessary and sufficient for a measure to be a BD-derivative locally, this is (a modern version of) the \emph{Saint-Venant theorem}, see for instance~\cite{AmroucheCiarletGratieKesavan06}.

Once this point of view is adopted, it is then natural to try  to understand the structure of the singular part of PDE-constrained measures. More precisely, given a linear homogeneous operator
\[
 \Acal := \sum_{\abs{\alpha} = k} A_\alpha \partial^\alpha,
\]
where \(A_\alpha \in \R^{m\times n}\),  \(\alpha=(\alpha_1,\ldots, \alpha_d) \in (\N \cup \{0\})^d\) is a multi-index, and \(\partial^\alpha:=\partial^{\alpha_1}_1\cdots\partial^{\alpha_d}_d\), we say that  an \(\R^n\)-valued (local) Radon measure \(\mu\in \Mcal_\loc(\Omega;\R^n)\) is \(\Acal\)-free if it satisfies 
\[
\Acal \mu=0  \qquad\text{in the sense of distributions.}
\] 
Note that since \(A_\alpha \in \R^{m\times n}\) this is actually a system of equations. A natural question is then to investigate the restrictions imposed on the singular part \(\mu^s\) of \(\mu\) by the differential constraint.
 
 To answer to this question, we first note that there are two trivial  instances: If \(\Acal = 0\), then no constraint is imposed. Conversely, if \(\Acal\) is \emph{elliptic}, i.e., if its symbol
 \begin{equation}\label{e:symbol}
 \Abb(\xi) := (2\pi\ii)^k \sum_{\abs{\alpha}=k} A_\alpha\xi^{\alpha}\in \R^{m\times n}
 \end{equation}
 is injective, then by the generalized Weyl lemma, \(\mu\) is smooth and thus no singular part is possible.
 
 In view of the above considerations it is natural to conjecture that the presence of singularities is related to the failure of ellipticity. This failure is measured by the \emph{wave cone} associated to \(\Acal\), first introduced by Murat and Tartar in the context of compensated compactness~\cite{Tartar79,Murat78,Murat79,Tartar83}:
\[
  \Lambda_\Acal := \bigcup_{\xi \in \Sbb^{d-1}} \ker \Abb(\xi).
\]
The main result of~\cite{DePhilippisRindler16} asserts that this cone is precisely what constrains the singular part of \(\mu\), see  also~\cite{DePhilippisMarcheseRindler17} and the surveys~\cite{DePhilippis18, DePhilippisRindler18} for  other applications of these results.

\begin{theorem}\label{thm:Afree}
Let \(\mu\in \Mcal(\Omega;\R^n)\)  be an \(\Acal\)-free measure, i.e.,
\[
\Acal \mu=0.
\]
Then, for \(|\mu^s|\)-almost all \(x\),
\[
\frac{\di \mu}{\di |\mu|}(x) \in   \Lambda_\Acal. 
\]
\end{theorem}

In the case  $\Acal = \Curl\Curl$ we obtain by direct computation, see~\cite[Example 3.10(e)]{FonsecaMuller99}, that for $M\in \Rdds$, $\xi \in \R^d$,
\[
-(4\pi)^{-2}\Abb(\xi)M=(M\xi)\otimes \xi+\xi\otimes (M\xi)-({\rm tr} M) \, \xi\otimes \xi -|\xi|^2 M,
\]
which gives
\[
 \ker \Abb(\xi) = \setb{a\otimes \xi+\xi \otimes a}{a\in \R^d}.
\]
Thus,
\begin{equation} \label{eq:Lambda}
 \Lambda_{\Curl\Curl} = \setb{a\odot b}{a,b\in \R^d}.
\end{equation}
Hence, Theorem~\ref{thm:Afree} implies Theorem~\ref{thm:AlbertiBD}. We remark that in the two-dimensional case $\mu \in \Mcal_\loc(\R^2;\R^{2 \times 2})$ we moreover have
\begin{equation} \label{eq:curlcurl2D}
\begin{gathered}
\Curl \Curl \mu = 0
\\
\Updownarrow
\\
 \curl \curl \mu = \partial_{22} \mu_{22} - \partial_{12} \mu_{12} - \partial_{12} \mu_{21} + \partial_{11} \mu_{22} = 0,
\end{gathered}
\end{equation}
where $\curl \, (\nu_1,\nu_2) := \partial_2 \nu_1 - \partial_1 \nu_2$ is the classical (scalar-valued) curl in two dimensions, applied row-wise.

Note also that if \(\Acal \colon \R^{\ell\times d} \to \R^{\ell\times d \times d}\) is the \(d\)-dimensional row-wise \(\curl\)-operator defined via
\[
(\curl A)_{ijk}:=\partial_j A_{ik}-\partial_k A_{ij},\qquad i=1,\dots, \ell, \quad j,k=1,\dots, d,
\]
one easily computes that 
\[
 \Lambda_{\curl} = \setb{a\otimes b }{a,b\in \R^d}.
\]
Hence, Theorem~\ref{thm:Afree} also provides a new proof of Alberti's rank one theorem.

Furthermore, we mention that in~\cite{ArroyoRabasaDePhilippisHirschRindler19} similar (more refined) techniques were used to recover the dimensional estimates and rectifiability results on the jump parts of BV- and BD-maps; we will discuss these results in Section~\ref{ssc:dim}.

\subsection*{Outline of the paper}
In Section~\ref{sec:rigidity} we start by showing some rigidity statements for maps whose symmetrized gradient is constrained to lie in a certain set. Parts of these result will be used later, but most importantly, we believe that they will give the reader a feel for how the differential constraint characterizing \(Eu\) can be used to understand BD-maps. In Section~\ref{sec:singularities} we give a sketch of the proof of Theorem~\ref{thm:AlbertiBD} and we outline how the improvements in~\cite{ArroyoRabasaDePhilippisHirschRindler19} give (optimal) dimensionality and rectifiability estimates. We also investigate the implications of Theorem~\ref{thm:AlbertiBD} on the structure of singularities of $\BD$-maps. Finally, in Section~\ref{sec:app}  we present, mostly without proofs,  some applications of the above results to the study of weak* lower semicontinuity of integral functionals, relaxation, and the characterization of Young measures generated by sequences of symmetrized gradients.

\subsection*{Acknowledgements}
This project has received funding from the European Research Council (ERC) under the European Union's Horizon 2020 research and innovation programme, grant agreement No 757254 (SINGULARITY).

\section{Rigidity}\label{sec:rigidity}

Before we come to more involved properties of general $\BD$-maps, we first investigate what can be inferred by using only elementary rigidity arguments. Besides being useful in the next section, these arguments are also instructive since they show the interplay between the \(\Curl\Curl\)-free condition and some pointwise properties, which is the main theme of this survey. Furthermore, the rigidity theorem, Theorem~\ref{thm:general rigidity}, will be used to study tangent measures later. Much of the discussion follows~\cite[Section 4.4]{Rindler11}. To make some proofs more transparent we start presenting the results in the two-dimensional case, where, however, all interesting effects are already present. At the end of the section we deal with the general case.

A \term{rigid deformation} is a skew-symmetric affine map $\omega \colon \R^d \to \R^d$, i.e., $u$ is of the form
\[
  \omega (x) = u_0 + \Xi x,  \qquad \text{where $u_0 \in \R^d$, $\Xi \in \R_\skw^{d \times d}$.}
\]
The following lemma is well-known and will be used many times in the sequel, usually without mentioning. We reproduce its proof here because the central formula~\eqref{eq:Wu_identity} will be of use later.

\begin{lemma} \label{lem:E_kernel}
The kernel of the linear operator $E \colon \BD_\loc(\R^d) \to \Mcal_\loc(\R^d;\R_\sym^{d \times d})$ given by
\[
Eu:=\frac{1}{2} \bigl(Du+Du^T  \bigr)
\]
  is the space of rigid deformations. 
  \end{lemma}

\begin{proof}
It is obvious that $E u$ vanishes for a rigid deformation $u$. For the other direction, let $u \in \BD_\loc(\R^d)$ with $Eu = 0$. Define
\[
Wu := \frac{1}{2} \bigl( Du - Du^T \bigr).
\]
Then, for all $i,j,k = 1,\ldots,d$, we have, in the sense of distributions, 
\begin{align}
  \partial_k (Wu)_{ij} &= \frac{1}{2} \bigl( \partial_{kj} u_i - \partial_{ki} u_j \bigr)  \notag\\
  &= \frac{1}{2} \bigl( \partial_{jk} u_i + \partial_{ji} u_k \bigr)
    - \frac{1}{2} \bigl( \partial_{ij} u_k + \partial_{ik} u_j \bigr)  \notag\\
  &= \partial_j (Eu)_{ik} - \partial_i (Eu)_{kj} \label{eq:Wu_identity}\\
  &= 0. \notag
\end{align}
As $Du = Eu + Wu$, this entails that $Du$ is a constant, hence $u$ is affine, and it is clear that it in fact must be a rigid deformation.
\end{proof}

It is an easy consequence of the previous lemma that any $u \in \BD_\loc(\R^d)$ with $Eu = S \Lcal^d$, where $S \in \R_\sym^{d \times d}$ is a fixed symmetric matrix, is an affine function. More precisely, $u(x) = u_0 + (S+\Xi)x$ for some $u_0 \in \R^d$ and $\Xi \in \R_\skw^{d \times d}$.

Next, we will consider what can be said about maps $u \in \BD_\loc(\R^d)$ for which
\begin{equation} \label{eq:incl}
  Eu = P \nu
\end{equation}
with a fixed matrix $P \in \Rdds$  and a measure \(\nu \in \Mcal_\loc(\R^d;\R)\). As we already saw in~\eqref{eq:Lambda}, a special role is played by the symmetric rank-one matrices, $a \odot b$ for $a,b \in \R^d$. We recall that those matrices can be characterized in terms of their eigenvalues:

\begin{lemma} \label{lem:sym_tensor_prod}
Let $M \in \R_\sym^{d \times d}$ be a non-zero symmetric matrix.
\begin{itemize}
  \item[(i)] If $\rank M = 1$, then $M = \pm a \odot a = \pm a \otimes a$ for a vector $a \in \R^d$.
  \item[(ii)] If $\rank M = 2$, then $M = a \odot b$ for vectors $a,b \in \R^d$ if and only if the two (non-zero, real) eigenvalues of $M$ have opposite signs.
\item[(iii)] If $\rank M \geq 3$, then $M$ cannot be written as $M = a \odot b$ for any vectors $a,b \in \R^d$.
\end{itemize}
\end{lemma}

\begin{proof}
\proofstep{Ad (i).} Every rank-one matrix $M$ can be written as a tensor product $M = c \otimes d$ for some vectors $c,d \in \R^d \setminus \{0\}$. By the symmetry, we get $c_i d_j = c_j d_i$ for all $i,j \in \{1,\ldots,d\}$, which implies that the vectors $c$ and $d$ are multiples of each other. We therefore find $a \in \R^d$ with $M = \pm a \otimes a$.

\proofstep{Ad (ii).} Assume first that $M = a \odot b$ for some vectors $a,b \in \R^d$. Clearly, \(M\) maps \(\spn\{a,b\}\) to itself and it is the zero map on the orthogonal complement, hence we may assume that \(d=2\).

Take an orthogonal matrix $Q \in \R^{2 \times 2}$ such that $QMQ^T$ is diagonal. We compute
\[
  QMQ^T = \frac{1}{2}Q \bigl( a \otimes b + b \otimes a \bigr) Q^T
  = \frac{1}{2} \bigl( Qa \otimes Qb + Qb \otimes Qa \bigr) = Qa \odot Qb,
\]
whence we may always assume without loss of generality that $M$ is already diagonal,
\[
  a \odot b = M = \begin{pmatrix} \lambda_1 &  \\  & \lambda_2 \end{pmatrix},
\]
where $\lambda_1,\lambda_2 \neq 0$ are the two eigenvalues of $M$. Writing this out componentwise, we get
\[
  a_1 b_1 = \lambda_1, \qquad a_2 b_2 = \lambda_2, \qquad a_1 b_2 + a_2 b_1 = 0.
\]
As $\lambda_1, \lambda_2 \neq 0$, also $a_1, a_2, b_1, b_2 \neq 0$, and hence
\[
  0 = a_1 b_2 + a_2 b_1 = \frac{a_1}{a_2} \lambda_2 + \frac{a_2}{a_1} \lambda_1.
\]
Thus, $\lambda_1$ and $\lambda_2$ must have opposite signs.

For the other direction, by transforming as before we may assume again that $M$ is diagonal:
\[
M=\sum_{i=1}^d \lambda_i v_i\otimes v_i,
\]
where \(\{v_i\}_i\) is an orthonormal basis of \(\R^d\). Since $\rank M = 2$, we know that  only two of the \(\lambda_i\) are non-zero. Hence, we can assume \(d=2\) and \(M\) to be diagonal, $M = \bigl( \begin{smallmatrix} \lambda_1 &  \\ & \lambda_2 \end{smallmatrix} \bigr)$, and that $\lambda_1$ and $\lambda_2$ do not have the same sign. Then, with $\gamma := \sqrt{-\lambda_1/\lambda_2}$,
we define
\[
  a := \begin{pmatrix} \gamma \\ 1 \end{pmatrix},  \qquad
  b := \begin{pmatrix} \lambda_1 \gamma^{-1} \\ \lambda_2 \end{pmatrix}.
\]
For $\lambda_1 > 0$, $\lambda_2 < 0$ say (the other case is analogous),
\[
  \lambda_1 \gamma^{-1} + \lambda_2 \gamma
  = \lambda_1 \sqrt{\frac{\abs{\lambda_2}}{\lambda_1}} - \abs{\lambda_2} \sqrt{\frac{\lambda_1}{\abs{\lambda_2}}}
  = 0,
\]
and therefore
\[
  a \odot b
  = \frac{1}{2} \begin{pmatrix} \lambda_1 & \lambda_2 \gamma \\ \lambda_1 \gamma^{-1} & \lambda_2 \end{pmatrix}
  + \frac{1}{2} \begin{pmatrix} \lambda_1 & \lambda_1 \gamma^{-1} \\ \lambda_2 \gamma & \lambda_2 \end{pmatrix}
  = \begin{pmatrix} \lambda_1 &  \\  & \lambda_2 \end{pmatrix}
  = M.
\]
This proves the claim.

\proofstep{Ad (iii).} This is trivial.
\end{proof}

In the remainder of this section, we will investigate in more detail two-dimensional BD-maps with fixed polar.  First, note that if $u \in \BD_\loc(\R^d)$, the map  \(\tilde u(x):= Q^{T}u(Qx)\), where \(Q\in \R^{d\times d}\), satisfies
\[
E\tilde u=Q^T Eu Q.
\] 
Hence, without loss of generality we may assume that $P$ in~\eqref{eq:incl} is  diagonal.

In the case $d=2$, according to Lemma~\ref{lem:sym_tensor_prod} we have three non-trivial cases to take care of, corresponding to the signs of the eigenvalues $\lambda_1$, $\lambda_2$; the trivial case $\lambda_1 = \lambda_2 = 0$, i.e., $P = 0$, was already settled in Lemma~\ref{lem:E_kernel}.

First, consider the situation that $\lambda_1, \lambda_2 \neq 0$ and that these two eigenvalues have opposite signs. Then, from (the proof of) Lemma~\ref{lem:sym_tensor_prod}, we know that $P = a \odot b$ ($a \neq b$) for
\[
  a := \begin{pmatrix} \gamma \\ 1 \end{pmatrix},  \qquad
  b := \begin{pmatrix} \lambda_1 \gamma^{-1} \\ \lambda_2 \end{pmatrix},
  \qquad\text{where}\qquad \gamma := \sqrt{-\frac{\lambda_1}{\lambda_2}}.
\]

The result about solvability of~\eqref{eq:incl} for this choice of $P$ is:

\begin{proposition}[Rigidity for $P = a \odot b$]
Let $P = \bigl( \begin{smallmatrix} \lambda_1 & \\ & \lambda_2 \end{smallmatrix} \bigr) = a \odot b$, where $\lambda_1, \lambda_2 \in \R$ have opposite signs. Then, there exists a map $u \in \BD_\loc(\R^2)$ solving the differential equation
\[
Eu = P\nu,  \qquad  \nu \in \Mcal_\loc(\R^2;\R),
\]
if and only if $\nu$ is of the form
\[
  \nu(\di x) = \mu_1(\di x \cdot a) + \mu_2(\di x \cdot b),
\]
where $\mu_1, \mu_2 \in \Mcal_\loc(\R)$. In this case, 
\begin{equation} \label{eq:P_symtens_u}
  u(x) = H_1(x \cdot a)b + H_2(x \cdot b)a + \omega(x), 
\end{equation}
with \(\omega\) a rigid deformation and $H_1, H_2 \in \BV_\loc(\R)$ satisfying $H_1' = \mu_1$ and $H_2' = \mu_2$.
\end{proposition}

Here, the notation $\mu_1(\di x \cdot a)$ denotes the measure $\gamma \in \Mcal_\loc(\R^2)$ that acts on Borel sets $B \subset \R^2$ as
\[
  \gamma(B) = \int_\R \mu_1 \bigl( B \cap (s a^\perp + \R a) \bigr) \dd s,
\]
where $a^\perp$ is a unit vector with $a \cdot a^\perp = 0$ (which is unique up to orientation). Likewise for $\mu_2(\di x \cdot b)$. Notice also that, since $a$ and $b$ are linearly independent, we could absorb the rigid deformation $r$ into $H_1$ and $H_2$.

\begin{proof}
By the chain rule in $\BV$ (see~\cite[Theorem~3.96]{AmbrosioFuscoPallara00book}), it is easy to deduce that all $u$ of the form~\eqref{eq:P_symtens_u} satisfy~\eqref{eq:incl} with $P = a \odot b$,  that is, $Eu = P\nu$ with $\nu \in  \Mcal_\loc(\R^2;\R)$.

For the other direction, we choose  \(Q\) to be an invertible matrix sending \(\{ \ee_1,  \ee_2\}\) to \(\{a,b\}\) and instead of $u$ work with $\tilde u(x):= Q^{T}u(Qx)$, for which
\[
  E \tilde{u} = \sqrt{2} (\ee_1 \odot \ee_2) \tilde{\nu} 
\]
with $\tilde{\nu} \in \Mcal_\loc(\R^2;\R)$. In the following we write simply $u$ in place of $\tilde{u}$.

We will use a slicing result~\cite[Proposition~3.2]{AmbrosioCosciaDalMaso97}, which essentially follows from Fubini's theorem: If for $\xi \in \R^2 \setminus \{0\}$ we define
\begin{align*}
  H_\xi &:= \setb{ x \in \R^2 }{ x \cdot \xi = 0 }, \\
  u_y^\xi(t) &:= \xi^T u(y + t\xi),  \qquad\text{where $t \in \R$, $y \in H_\xi$,}
\end{align*}
then the result in \emph{loc.~cit.\ }states
\begin{equation} \label{eq:Eu_slicing}
  \absb{\xi^T Eu \xi} = \int_{H_\xi} \absb{Du_y^\xi} \dd \Hcal^1(y)
  \qquad\text{as measures.}
\end{equation}
We have $Eu = \sqrt{2} (\ee_1 \odot \ee_2) \nu$, so if we apply~\eqref{eq:Eu_slicing} for $\xi = \ee_1$, we get
\[
  0 = \sqrt{2} \, \absb{\ee_1^T (\ee_1\odot \ee_2) \ee_1} \, \abs{\nu}
  = \int_{H_\xi} \absb{\partial_t u_1(y + t\ee_1)} \dd \Hcal^1(y),
\]
where we wrote $u = (u_1,u_2)$. This yields $\partial_1 u_1 = 0$ distributionally, whence $u_1(x) = H_2(x_2)$ for some $H_2 \in \Lrm_\loc^1(\R)$. Analogously, we find that $u_2(x) = H_1(x_1)$ with $H_1 \in  \Lrm_\loc^1(\R)$. Thus, we may decompose
\[
  u(x) = \begin{pmatrix} 0 \\ H_1(x_1) \end{pmatrix}
    + \begin{pmatrix} H_2(x_2) \\ 0 \end{pmatrix}
  = H_1(x \cdot \ee_1)\ee_2 + H_2(x \cdot \ee_2)\ee_1,
\]
and it only remains to show that $H_1, H_2 \in \BV_\loc(\R)$. For this, fix $\eta \in \Crm_c^1(\R;[-1,1])$ with $\int \eta \dd t = 1$ and calculate for all $\phi \in \Crm_c^1(\R;[-1,1])$ by Fubini's Theorem,
\begin{align*}
  2 \int \phi \otimes \eta \dd(Eu)_{12} &= - \int u_2 (\phi' \otimes \eta) \dd x - \int u_1 (\phi \otimes \eta') \dd x \\
  &= - \int H_1 \phi' \dd x_1 \cdot \int \eta \dd x_2 - \int u_1 (\phi \otimes \eta') \dd x.
\end{align*}
So, with $K := \supp \phi \times \supp \eta$,
\[
  \absBB{ \int H_1 \phi' \dd x } \leq 2 \abs{Eu}(K) + \norm{u_1}_{\Lrm^1(K)} \cdot \norm{\eta'}_\infty < \infty
\]
for all $\phi \in \Crm_c^1(\R)$ with $\norm{\phi}_\infty \leq 1$, hence $H_1 \in \BV_\loc(\R)$. Likewise, $H_2 \in \BV_\loc(\R)$, and we have shown the proposition.
\end{proof}

In the case $\lambda_1 \neq 0$, $\lambda_2 = 0$, i.e., $P = \lambda_1 (\ee_1 \odot \ee_1)$, one could guess by analogy to the previous case that if $u \in \BD_\loc(\R^2)$ satisfies $Eu = P\nu$ for some $\nu \in \Mcal_\loc(\R)$, then $u$ and $\nu$ should only depend on $x_1$ up to a rigid deformation. This, however, is \emph{false}, as can be seen from the following example.

\begin{example}
Consider
\[
  P := \begin{pmatrix} 1 & \\ & 0 \end{pmatrix},  \qquad
  u(x) := \begin{pmatrix} 4 x_1^3 x_2 \\ -x_1^4 \end{pmatrix},  \qquad
  g(x) := 12x_1^2 x_2.
\]
Then, $u$ satisfies $E u = Pg \, \Lcal^d$, but neither $u$ nor $g$ only depend on $x_1$.
\end{example}

The general statement reads as follows.

\begin{proposition}[Rigidity for $P = a \odot a$]
Let $P = \bigl( \begin{smallmatrix} \lambda_1 & \\ & 0 \end{smallmatrix} \bigr) = \lambda_1 (\ee_1 \odot \ee_1)$. Then, there exists a map $u \in \BD_\loc(\R^2)$ solving the differential equation
\[
  E u = P\nu,  \qquad  \nu \in \Mcal_\loc(\R^2;\R),
\]
if and only if $\nu$ is of the form
\[
  \nu(\di x) = \mu(\di x_1) + \gamma(\di x_1) \otimes \bigl( x_2 \, \Lcal^1(\di x_2) \bigr),
\]
where  $\mu, \gamma\in \Mcal_\loc(\R)$. In this case,
\[
  u(x) =  \lambda_1 \begin{pmatrix} H(x_1) + \Pcal'(x_1)x_2 \\ -\Pcal(x_1) \end{pmatrix} + \omega(x),
\]
with $\omega$ a rigid deformation  and $H \in \BV_\loc(\R)$, $\Pcal \in \Wrm^{1,\infty}_\loc(\R)$ with $\Pcal' \in \BV_\loc(\R)$ satisfying $H' = \mu$ and $\Pcal'' = \gamma$.
\end{proposition}

\begin{proof}
The necessity is again a simple computation.

For the sufficiency, assuming by a mollification argument that $u$ is smooth, there exists $g \in \Crm^\infty(\R^2)$ such that
\[
  \Ecal u = \lambda_1(\ee_1 \odot \ee_1) g  \qquad\text{and}\qquad
  E^s u = 0.
\]
We have from~\eqref{eq:Wu_identity} that
\[
  \partial_k (W u)_{ij} = \partial_j (E u)_{ik} - \partial_i (E u)_{kj}
  \qquad\text{for $i,j,k = 1,2$.}
\]
Thus,
\[
  \nabla (Wu)_{12} = (\lambda_1 \partial_2 g,0).
\]
This gives that $(W u)_{12}$ and hence also $\partial_2 g$ depend on the first component $x_1$ of $x$ only, $\partial_2 g(x) = p(x_1)$ say. Define
\[
  h(x) := g(x) - p(x_1) x_2
\]
and observe that $\partial_2 h = 0$. Hence we may write $h(x) = h(x_1)$ and have now decomposed $g$ as
\[
  g(x) = h(x_1) + p(x_1) x_2.
\]
This gives the claimed decomposition in the smooth case. The general case follows by approximation.
\end{proof}

Finally, we consider the case where the eigenvalues $\lambda_1$ and $\lambda_2$ are non-zero and have the same sign. Then, $P \neq a \odot b$ for any $a,b \in \R^2$ by Lemma~\ref{lem:sym_tensor_prod}. Define the differential operator
\[
  \Acal_P := \lambda_2 \partial_{11} + \lambda_1 \partial_{22}
\]
and notice that whenever a function $g \colon \R^2 \to \R$ satisfies $\Acal_P g = 0$ distributionally, the function \(\tilde g(x_1,x_2):=g(\sqrt{\abs{\lambda_2}} x_1, \sqrt{\abs{\lambda_1} }x_2)\) is harmonic (recall that \(\lambda_1, \lambda_2\) have the same sign). In particular,  by Weyl's lemma, $g$ is smooth.

\begin{proposition}[Rigidity for $P \neq a \odot b$] \label{prop:P_neq_a_odot_b_2D}
Let $P = \bigl( \begin{smallmatrix} \lambda_1 & \\ & \lambda_2 \end{smallmatrix} \bigr)$, where $\lambda_1, \lambda_2 \in \R$ have the same sign. Then, there exists a map $u \in \BD_\loc(\R^2)$ solving the differential equation
\[
  E u = P\nu,  \qquad  \nu \in \Mcal_\loc(\R^2;\R),
\]
if and only if $\nu$ satisfies
\[
  \Acal_P \nu = 0.
\]
Moreover, in this case both $\nu$ and $u$ are smooth.
\end{proposition}

\begin{proof}
First assume that $g \in \Crm^\infty(\R^2)$ satisfies $\Acal_P g = 0$. Define 
\[
   F := (- \lambda_1 \partial_2 g, \lambda_2 \partial_1 g)
\]
and observe 
\[
  \curl F = - \lambda_1 \partial_{22} g - \lambda_2 \partial_{11} g = - \Acal_P g = 0.
\]
Hence, there exists $f \in \Crm^\infty(\R^2)$ with $\nabla f = F$, in particular
\begin{equation} \label{eq:solv_cond}
  \partial_1 f = - \lambda_1 \partial_2 g,  \qquad  \partial_2 f = \lambda_2 \partial_1 g.
\end{equation}
Put
\[
  \Ucal := \begin{pmatrix} \lambda_1 & 0 \\ 0 & \lambda_2 \end{pmatrix} g
  + \begin{pmatrix} 0 & -1 \\ 1 & 0 \end{pmatrix} f.
\]
We calculate (we apply the curl row-wise), using~\eqref{eq:solv_cond},
\[
  \curl \, \Ucal = \begin{pmatrix} \curl \, \bigl( \lambda_1 g, -f \bigr) \\
    \curl \, \bigl( f, \lambda_2 g \bigr) \end{pmatrix}
  = \begin{pmatrix} \lambda_1 \partial_2 g + \partial_1 f \\ 
    \partial_2 f - \lambda_2 \partial_1 g \end{pmatrix} = 0.
\]
Let $u \in \Crm^\infty(\R^2;\R^2)$ be such that $\nabla u = \Ucal$. Then, as distributions,  $E u = P g$.

For the other direction, it suffices to show that $Eu = P\nu$ for some $\nu \in \Mcal_\loc(\R^2)$ implies $\Acal_P \nu = 0$. The smoothness of $u,\nu$ then follows from Weyl's lemma as remarked above.  Since \(d=2\) we can exploit~\eqref{eq:curlcurl2D} to get that
\begin{equation}\label{e:ellipticg}
0=\curl \curl (Eu )=\curl \curl \biggl[ \begin{pmatrix} \lambda_1 & 0 \\ 0 & \lambda_2 \end{pmatrix} \nu \biggr]=\Acal_P \nu,
\end{equation}
so that the claim follows.
\end{proof}

\begin{remark}\label{rmk:ellipticity}
Note that the key point in the above lemma is that  whenever \(Eu=P\nu\) with \(P\ne a\odot b\) for any $a,b \in \R^2$, the fact that \(Eu\) is \(\curl\curl\)-free implies that the  measure \(\nu\) is actually a solution of an \emph{elliptic} PDE, namely~\eqref{e:ellipticg}. This is also the key fact underlying the proof of Theorem~\ref{thm:AlbertiBD} in the next section.
\end{remark}

\begin{remark}[Comparison to gradients]
Proposition~\ref{prop:P_neq_a_odot_b_2D} should be contrasted with the corresponding situation for gradients. If $u \in \Wrm_\loc^{1,1}(\R^2;\R^2)$ satisfies
\[
  \nabla u \in \spn \{ P \}  \qquad\text{pointwise a.e.,}
\]
and $\rank P = 2$, then necessarily $u$ is affine, a proof of which can be found, for instance, in~\cite[Lemma~3.2]{Rindler12} (this rigidity result is closely related to Hadamard's jump condition, also see~\cite[Proposition~2]{BallJames87},~\cite[Lemma~1.4]{DeLellis08},~\cite[Lemma~2.7]{Muller99} for related results). Notice that this behavior for the gradient is in sharp contrast to the behavior for the symmetrized gradient, as can be seen from the following example.
\end{remark}

\begin{example} \label{ex:weak_rigidity}
Let
\[
  P := \begin{pmatrix} 1 & \\ & 1 \end{pmatrix},  \qquad
  u(x) := \begin{pmatrix} \ee^{x_1} \sin(x_2) \\ -\ee^{x_1} \cos(x_2) \end{pmatrix},  \qquad
  g(x) := \ee^{x_1} \sin(x_2).
\]
Then, one can check that $g$ is harmonic (corresponding to $\Acal_P g = \Delta g = 0$) and $u$ satisfies $\Ecal u = Pg$. So, the fact that $P$ cannot be written as a symmetric tensor product does not imply that any solution to the differential inclusion $\Ecal u \in \spn \{P\}$ must be affine. However, as noted in Remark~\ref{rmk:ellipticity}, \(g\) is still ``rigid'' (in a weaker sense) since it has to satisfy an elliptic PDE.
\end{example}

We conclude this section with the following general version of the rigidity statements in every dimension; the proofs of~(i),~(ii) follow the same (elementary) strategy as above, whereas in~(iii) we see the first instance of an approach via the Fourier transform.
\begin{theorem}\label{thm:general rigidity}
Let \(u\in \BD_\loc (\R^d)\) and assume that
\[
Eu = P \nu
\]
for a fixed matrix $P \in \Rdds$  and a (signed) measure $\nu \in \Mcal_\loc(\R^d;\R)$. Then:
 \begin{itemize}
 \item[(i)] If $P = a \odot b$ for some $a,b \in \R^d$ with \(a \ne \pm b\), then there exist two functions \(H_1,H_2\in \BV_\loc(\R)\), a vector  \(v \in \spn\{a,b\}^\perp\), and a rigid deformation \(\omega\) such that 
\begin{align*}
 u(x)&=a\bigl(H_1(x\cdot b)+(x\cdot b)(x\cdot v)\bigr)+b \bigl(H_2(x\cdot a)+(x\cdot a)(x\cdot v)\bigr)
 \\
 &\qquad - v \, (x\cdot a)(x\cdot b) +\omega(x).
\end{align*}
 \item[(ii)] If $P = \pm a \odot a$ for some $a \in \R^d$, then there exist a function \(H\in \BV_\loc(\R)\), an orthonormal basis \(\{v_2,\ldots, v_j\}\) of \(\spn\{a\}^\perp\), functions \(\Pcal_j\in \Wrm^{1,\infty}_\loc(\R)\) with \(\Pcal'_j\in \BV_\loc(\R)\) ($j = 2,\ldots,d$), and a rigid deformation \(\omega\) such that 
 \[
 u(x)= a\biggl(H(x\cdot a)+\sum^d_{j=2}(x\cdot v_j) \Pcal'_j(x\cdot a)\biggr)-\sum^d_{j=2} v_j \Pcal_j(x\cdot a)+\omega(x).
 \]
\item[(iii)] If $P \neq a \odot b$ for any $a,b \in \R^d$, then $u$ and $\nu$ are smooth.
 \end{itemize}
\end{theorem}

\begin{proof}

\proofstep{Ad (i).}
By regularization we can assume that \(u\) is smooth and that 
\[
\Ecal u= 2(a\odot b) g,  \qquad g \in \Crm^\infty(\R^d).
\] 
Recall from~\eqref{eq:Wu_identity} that for $Wu := \frac{1}{2} ( Du - Du^T )$ we have
\[
  \partial_k (W u)_{ij} = \partial_j (E u)_{ik} - \partial_i (E u)_{kj},
  \qquad\text{for $i,j,k = 1,\ldots,d$.}
\]
We assume without loss of generality that \(a=\ee_1\) and \(b=\ee_2\). Then,
\begin{align*}
\nabla  (Wu)_{12}&=-\nabla  (Wu)_{21}=(-\partial_1g,\partial _2 g,0,\ldots,0), && \\
\nabla  (Wu)_{1j}&=-\nabla  (Wu)_{j1}=(0,\partial_j g,0,\ldots,0)&& \text{for all \(j\geq 3\),}
\\
\nabla  (Wu)_{2j}&=-\nabla  (Wu)_{j2}=(\partial_j g,0,\ldots,0) &&\text{for all \(j\geq 3\),} \\
\nabla  (Wu)_{ij}&=0\qquad && \text{for all \(i,j\geq 3\).}
\end{align*}
From this we readily  deduce that \(\partial_j g = \mathrm{const}\) for \(j=3,\ldots, d\) and, applying the curl to the first equation, that
\[
\partial_{12}g = 0.
\]
Hence, we can write 
\begin{equation}\label{e:adolfo}
g(x)=\frac{h_1(x_2)+h_2(x_1)}{2}+(x\cdot v),
\end{equation}
where $h_1, h_2 \in \Crm^\infty(\R)$ and \(v\) is orthogonal to \(\spn\{\ee_1,\ee_2\}\).

We may compute that for the $u$ given in~(i) with $H_i$ defined via $H_i' = h_i$, $i=1,2$ (the shift of $H_i$ is arbitrary and can later be absorbed into the rigid deformation $r$) and $a := \ee_1$, $b := \ee_2$, we have that $\Ecal u= 2(a\odot b) g$ with the $g$ above. Thus, by Lemma~\ref{lem:E_kernel}, we conclude that our $u$ must have this form (we absorb a rigid deformation into $\omega$).

\proofstep{Ad (ii).}
We assume that \(P=\ee_1 \odot \ee_1\) and we  argue as above to deduce that
\[
\nabla  (Wu)_{1j}=-\nabla  (Wu)_{1j}=(\partial_j g,0,\ldots,0)\qquad \text{for all $j\geq 2$,}
\]
and \(\nabla (Wu)_{ij}=0\) if \(i,j \geq 2\). This implies that  
\[
\partial_j g(x)=p_j(x_1) \qquad\text{for all \(j\ge 2\)}
\]
for suitable functions \(p_j \in \Crm^\infty(\R)\).  Hence,
\begin{equation}\label{e:adolfo2}
2g(x)=h(x_1)+\sum_{j=2}^d x_j p_j(x_1)
\end{equation}
for some $h \in \Crm^\infty(\R)$. Again, defining $H$ via $H' = h$ and $\Pcal_j$ via $\Pcal_j'' = p_j$, we obtain for the $u$ given in~(ii) that $\Ecal u = 2(\ee_1 \odot \ee_1)g$ with $g$ as above. We conclude as before via Lemma~\ref{lem:E_kernel}.

\proofstep{Ad (iii).}
Let $L := \spn\{P\}$ and denote by $\Pbf \colon \Rdds \to \Rdds$ the orthogonal projection onto the orthogonal complement $L^\perp$ of $L$. For every smooth cut-off function $\rho \in \Crm^\infty_c(\R^d;[0,1])$ with $\rho \equiv 1$ on a bounded open set $U \subset \R^d$, the function $w := \rho u$ satisfies
\[
  \Ecal w = \rho \Ecal u + u \odot \nabla \rho.
\]
So,
\begin{equation} \label{eq:PDuf}
  \Pbf(\Ecal w) = \Pbf(u \odot \nabla \rho) =: R \in \Lrm^p(\R^d;\Rdds)
\end{equation}
with $p = d/(d-1)$ by the embedding $\BD_\loc(\R^d) \embed \Lrm^{d/(d-1)}_\loc(\Omega;\R^d)$~\cite{TemamStrang80}. 

Applying the Fourier transform (which we define for an integrable function $w$ via $\hat{w}(\xi) := \int w(x) \ee^{2\pi\ii x \cdot \xi} \dd x$) to both sides of~\eqref{eq:PDuf} and considering $\Pbf$ to be identified with its complexification (that is, $\Pbf(A+\ii B) = \Pbf(A) + \ii \Pbf(B)$ for $A,B \in \Rdds$), we arrive at
\[
  \Pbf(\widehat{\Ecal w}(\xi)) = (2\pi\ii) \, \Pbf(\hat{w}(\xi) \odot \xi) = \hat{R}(\xi).
\]
Here, we used that for a symmetrized gradient one has
\[
\widehat{\Ecal w}(\xi) =(2\pi \ii) \,\hat{w}(\xi)\odot \xi,\qquad \xi \in \R^d.
\]

The main point is to show (see below) that we may \enquote{invert} $\Pbf $ in the sense that if
\begin{equation} \label{eq:PnablawR}
  \Pbf(\widehat{\Ecal w}) = \hat{R}
\end{equation}
for some $w \in \Wrm^{1,p}(\R^d;\R^m)$, $R \in \Lrm^p(\R^d;L^\perp)$, then
\begin{equation} \label{eq:qc_multiplier}
  \widehat{\Ecal w}(\xi)
  = \Mbf(\xi) \hat{R}(\xi),  \qquad
  \xi \in \R^d \setminus \{0\},
\end{equation}
for some family of linear operators $\Mbf(\xi) \colon \Rdds \to \Rdds$ that depend smoothly and positively \mbox{$0$-homogeneously} on $\xi$.

We then infer from the Mihlin multiplier theorem (see for instance~\cite[Theorem~5.2.7]{Grafakos14book1}) that
\[
  \norm{\Ecal w}_{\Lrm^p}
  \leq C \norm{\Mbf}_{\Crm^{\floor{d/2}+1}} \norm{R}_{\Lrm^p}
  \leq C \norm{u}_{\Lrm^p}.
\]
So, also using $\rho \Ecal u = \Ecal w - u \odot \nabla \rho$, we get the estimate
\[
  \norm{\Ecal u}_{\Lrm^p(U)}
  \leq \norm{\Ecal w}_{\Lrm^p(\Omega)} + \norm{u \odot \nabla \rho}_{\Lrm^p(\Omega)}
  \leq C \norm{u}_{\Lrm^p(\Omega)}
\]
for some constant $C > 0$. In particular, by Korn's inequality~\eqref{e:korn}, \(u \in \Wrm_\loc^{1,p}(\Omega;\R^d) \subset\Lrm^{p^*}(\Omega;\R^d)\) for $p^* := dp/(d-p)$ if $d<p$ and $p^*=\infty$ if $p>d$). We can now iterate (``bootstrap'') via~\eqref{eq:PDuf} (which we also need to differentiate in order to get bounds on derivatives) to conclude that $u$ is smooth.

It remains to show~\eqref{eq:qc_multiplier}. Notice that $\Pbf(a \odot \xi) \neq 0$ for any $a \in \C^m \setminus \{0\}$, $\xi \in \R^d \setminus \{0\}$ by the assumption  on \(P\). Thus, for some constant $C > 0$ we have the \emph{ellipticity} estimate
\[
  \abs{a \odot \xi} \leq C \abs{\Pbf(a \odot \xi)}  \qquad
  \text{for all $a \in \C^m$, $\xi \in \R^d$.}
\]
The (complexified) projection $\Pbf \colon \C^{m \times d} \to \C^{m \times d}$ has kernel $L^{\C} := \spn_{\C} L$ (the complex span of $L$), which in the following we also denote just by $L$. Hence, $\Pbf$ descends to the quotient
\[
  [\Pbf] \colon \C^{m \times d} / L \to \ran \Pbf,
\]
and $[\Pbf]$ is an invertible linear map. For $\xi \in \R^d \setminus \{0\}$ let
\[
  \bigl\{ F, \ee_1 \odot \xi, \ldots, \ee_d \odot \xi, G_{d+1}(\xi),
  \ldots, G_{d^2-1}(\xi) \bigr\}
\]
be a $\C$-basis of $\C^{m \times d}$ with the property that the matrices $G_{d+1}(\xi), \ldots, G_{d^2-1}(\xi)$ depend smoothly on $\xi$ and are positively $1$-homogeneous in $\xi$, that is, $G_{d+1}(\alpha \xi) = \alpha G_{d+1}(\xi)$ for all $\alpha \geq 0$. Furthermore, for $\xi \in \R^d \setminus \{0\}$ denote by $\Qbf(\xi) \colon \C^{m \times d} \to \C^{m \times d}$ the (non-orthogonal) projection with
\begin{align*}
  \ker \Qbf(\xi) &= L, \\
  \ran \Qbf(\xi) &= \spn \bigl\{ \ee_1 \odot \xi, \ldots, \ee_d \odot \xi, G_{d+1}(\xi),
  \ldots, G_{d^2-k}(\xi) \bigr\}.
\end{align*}
If we interpret $\ee_1 \odot \xi, \ldots, \ee_d \odot \xi, G_{d+1}(\xi), \ldots, G_{d^2-1}(\xi)$ as vectors in $\R^{d^2}$ and collect them into the columns of the matrix $X(\xi) \in \R^{d^2 \times (d^2-1)}$, and if we further let $Y \in \R^{d^2 \times (d^2-1)}$ be a matrix whose columns comprise an orthonormal basis of $L^\perp$, then, up to a change in sign for one of the $G_l$'s, there exists a constant $c > 0$ such that
\[
  \det (Y^T X(\xi)) \geq c > 0, \qquad
  \text{for all $\xi \in \Sbb^{d-1}$.}
\]
Indeed,
 if $\det (Y^T X(\xi))$ was not uniformly bounded away from zero for all $\xi \in \Sbb^{d-1}$, then by compactness there would exist a $\xi_0 \in \Sbb^{d-1}$ with $\det (Y^T X(\xi_0)) = 0$, a contradiction. We can then write $\Qbf(\xi)$ explicitly as
\[
  \Qbf(\xi) = X(\xi)(Y^T X(\xi))^{-1}Y^T.
\]
This implies that $\Qbf(\xi)$ depends positively $0$-homogeneously and smoothly on $\xi \in \R^d \setminus \{0\}$. Also $\Qbf(\xi)$ descends to the quotient
\[
  [\Qbf(\xi)] \colon \C^{m \times d} / L \to \ran \Qbf(\xi),
\]
which is now invertible. It is not difficult to see that $\xi \mapsto [\Qbf(\xi)]$ is still positively $0$-homogeneous and smooth in $\xi \neq 0$ (by utilizing the basis given above). Since $\hat{w}(\xi) \odot \xi \in \ran \Qbf(\xi)$, we have
\[
  [\Qbf(\xi)]^{-1}(\hat{w}(\xi) \odot \xi) = [\hat{w}(\xi) \odot \xi],
\]
where $[\hat{w}(\xi) \odot \xi]$ designates the equivalence class of $\hat{w}(\xi) \odot \xi$ in $\C^{m \times d} / L$. This fact in conjunction with $\widehat{\Ecal w}(\xi) = (2\pi\ii) \, \hat{w}(\xi) \odot \xi$ allows us to rewrite~\eqref{eq:PnablawR} in the form
\[
  (2\pi\ii) \, [\Pbf] [\Qbf(\xi)]^{-1}(\hat{w}(\xi) \odot \xi) = \hat{R}(\xi),
\]
or equivalently as
\[
  \widehat{\Ecal w}(\xi) = (2\pi\ii) \, \hat{w}(\xi) \odot \xi = [\Qbf(\xi)] [\Pbf]^{-1} \hat{R}(\xi).
\]
The multiplier $\Mbf(\xi) \colon \Rdds \to \Rdds$ for $\xi \in \R^d \setminus \{0\}$ is thus given by
\[
  \Mbf(\xi) := [\Qbf(\xi)] [\Pbf]^{-1},
\]
which is smooth and positively $0$-homogeneous in $\xi$. Consequently, we have shown the multiplier equation~\eqref{eq:qc_multiplier}.
\end{proof}

\section{Singularities}  \label{sec:singularities}

In this section we sketch the proof of Theorem~\ref{thm:AlbertiBD} and present some of its implications concerning the structure of singularities that can occur in BD-maps. We will also outline how this type of argument allows one to recover the dimensionality results in Theorem~\ref{thm:BdstructIntro}.

\subsection{Proof sketch of Theorem~\ref{thm:AlbertiBD}}

To simplify the proof and to expose the main ideas as clearly as possible we assume again that we are working in dimension \(d=2\).  Our argument for BD-maps here is a bit more direct than the original one in~\cite{DePhilippisRindler16} and does not use Fourier analysis. We also make the connection to the rigidity results of Section~\ref{sec:rigidity} explicit. This stresses the crucial argument, namely to exploit the ellipticity contained in the condition \(\frac{\di \mu}{\di \abs{\mu}}(x_0) \notin \Lambda_{\Acal}\). Let us also note that by using the slicing properties of BD-maps~\cite[Proposition 3.4]{AmbrosioCosciaDalMaso97} and by arguing as in~\cite{Alberti93} (see also~\cite[Section~2]{DeLellis08}) one can recover the theorem in any dimension from this particular case.

We assume by contradiction that the set 
\[
E:=\setBB{ x\in \R^2 }{ \frac{\di E u}{\di \abs{Eu}}(x) \ne a \odot b \text{ for any $a,b \in \R^2$} }
\]
satisfies \(|E^s u|(E)>0\). We now want to zoom in around a generic  point \(x_0\in E\). To this end we recall the notion of tangent measure: For a vector-valued  Radon measure $\mu \in \Mcal_\loc(\R^d;\R^n)$ and $x_0 \in \R^d$, a \term{tangent measure} to $\mu$ at $x_0$ is any (local) weak* limit in the space $\Mcal_\loc(\R^d;\R^n)$ of the rescaled measures
\[
\mu^{x_0,r_k}:=c_k T^{x_0,r_k}_\# \mu 
\]
for some sequence $r_k \todown 0$ of radii and some sequence $c_k > 0$ of rescaling constants. The definition of the push-forward $T^{x_0,r_k}_\# \mu$ here expands to
\[
  [T^{x_0,r_k}_\# \mu](B) := \mu(x_0 + r_k B) \qquad
  \text{for any Borel set $B \subset \R^2$.}
\]
We denote by \(\Tan(\mu, x_0)\) the set of all possible tangent measures of $\mu$ at \(x_0\). It is a remarkable theorem of Preiss~\cite{Preiss87} (see, e.g.,~\cite[Proposition~10.5]{Rindler18book} for a proof in our notation) that for every measure \(\mu\), 	\(\Tan(\mu, x_0)\) contains at least one non-zero measure for \(\abs{\mu}\)-almost every \(x_0\). Furthermore, at \(\abs{\mu}\)-almost all points $x_0$,
\begin{equation} \label{eq:Tan_polar}
\Tan(\mu, x_0)=\frac{\di \mu}{\di \abs{\mu}}(x_0) \cdot \Tan(\abs{\mu}, x_0),	
\end{equation}
see~\cite[Lemma~10.4]{Rindler18book}. If one assumes that  \(|E^su|(E)>0\), it then follows by elementary arguments from measure theory, see, e.g.,~\cite[Proof of Theorem 1.1]{DePhilippisRindler16}  that there exists at least one point \(x_0\in E\) and a sequence of radii \(r_k\todown 0\) such that the following properties hold:
\begin{itemize}
\item[(i)] $\displaystyle \lim_{k\to \infty} \frac{|E^a u|(B_{r_k}(x_0))}{|E^s u|(B_{r_k}(x_0))}=0$; 
\item[(ii)] $\displaystyle \lim_{k\to \infty}  \dashint_{B_{r_k}(x_0)} \biggl|\frac{\di Eu }{\di |Eu|}(x)-\frac{\di Eu }{\di |Eu|}(x_0)\biggr|\dd|E^s u|(x)=0$;
\item[(iii)] there exists a positive Radon measure \(\sigma \in \Tan(\abs{E^s u},x_0)\) with $\sigma \restrict B_{1/2} \neq 0$ ($B_{1/2} := B_{1/2}(0)$) and such that 
\[
\sigma_k:=\frac{T^{x_0,r_k}_\#|E^s u|}{|E^s u|(B_{r_k}(x_0))}\toweakstar \sigma  \qquad\text{in $\Mcal_\loc(\R^2)$;}
\]
\item[(iv)] $\displaystyle P:= \frac{\di Eu}{\di |Eu|}(x_0)\ne a \odot b$ for any $a,b \in \R^d$.
\end{itemize}
Define  
\[
v_k(y):= \frac{r_k^{d-1}}{|E^s u|(B_{r_k}(x_0))} \, u(x_0+r_k y),  \qquad y \in \R^2.
\]
We have the following Poincar\'{e}-type inequality in BD, proved in~\cite{TemamStrang80}:
\[
  \inf_{\text{$\omega$ rigid deformation}} \norm{u + \omega}_\BD \lesssim |Eu|(\Omega),  \qquad u \in \BD(\Omega).
\]
Thus, we conclude that there exists a sequence of rigid deformations \(\omega_k\) and a map \(v\in \BD_\loc(\R^2)\) such that 
\[
(v_k+\omega_k) \toweakstar v \quad \text{in \(\BD_\loc(\R^2)\)}
\]
and \(v\) satisfies 
\[
E v=P \sigma \qquad\text{with}\qquad \text{$P \ne a \odot b$ for any $a,b \in \R^2$,}
\]
and \(\sigma=\abs{Ev}\) is a positive measure. By Proposition~\ref{prop:P_neq_a_odot_b_2D}, \(\sigma\) is smooth. Unfortunately, this is however not in contradiction with $\sigma \in \Tan(|E^su|,x_0)\setminus\{0\}$, since there are purely singular measures having only Lebesgue-absolutely continuous measures as tangents at almost all points, see~\cite[Example~5.9~(1)]{Preiss87}. In order to prove the theorem we thus have to exploit the ellipticity mentioned in Remark~\ref{rmk:ellipticity} in a more careful way.

Let us assume without loss of generality that \(P=\bigl(\begin{smallmatrix} 1 & \\ & 1 \end{smallmatrix}\bigr)\), so that $\Acal_P$ defined in Proposition~\ref{prop:P_neq_a_odot_b_2D} is the Laplace operator. By recalling that
\[
\Curl \Curl Ev_k=0
\]
we can use~\eqref{eq:curlcurl2D} to get, cf.~\eqref{e:ellipticg}, 
\begin{equation}\label{e:sigmak}
\Delta \sigma_k=\curl\curl \, (P\sigma_k) =\curl\curl \, (P\sigma_k-Ev_k).
\end{equation}
Furthermore, by combining (i) and (ii) above it is not hard to check that 
\[
\lim_{k\to \infty} \abs{Ev_k-P\sigma_k}(B_1)=0.
\]
We now take a cut-off function  \(\varphi \in \Crm_c^\infty(B_1;[0,1])\) with \(\varphi \equiv 1\) on \(B_{1/2}\). Exploiting the identity (in the sense of distributions)
\[
\partial_{ii} (\varphi \nu) = \varphi \partial_{ii} \nu + 2\partial_i(\partial _i \varphi \nu) - \nu \partial_{ii} \varphi,
\]
which is valid for any smooth  function \(\varphi\) and any measure \(\nu\), we get, using~\eqref{e:sigmak}, that
\[
\Delta (\varphi\sigma_k) = \phi \curl\curl Z_k +\Div R_k+S_k
\]
where \(Z_K\), \(R_k\), and \(S_k\) are measures supported in \(B_1\) and satisfying
\[
|Z_k|(B_1)\to 0, \qquad \sup_{k} \, \bigl( |R_k|(B_1)+|S_k|(B_1) \bigr) \lesssim 1.
\]
We apply  \(\Delta^{-1}\) to both sides of the above equation to get
\[
\begin{split}
\varphi\sigma_k&= \Delta^{-1}(\phi \curl\curl Z_k)+\Delta^{-1}\Div R_k+\Delta^{-1} S_k
\\
&=K_1 \conv Z_k+K_2 \conv R_k+K_3 \conv S_k,
\end{split}
\]
where 
\[
K_3(x)=\frac{1}{2\pi} \log|x|, \qquad K_2=DK_3,
\]
and \(K_1\) is a constant-coefficient polynomial in the second derivatives of \(K_3\). In particular, \(K_1\) is a Calder\'{o}n--Zygmund kernel and (see~\cite{Stein93book,Grafakos14book1})
\[
\abs{K_2}(x)\lesssim |x|^{-1},\qquad \abs{K_3}(x)\lesssim |\ln{|x|}|.
\]
By this and standard estimates~\cite{Stein93book,Grafakos14book1}, one easily sees that the sequences \((K_2 \conv R_k)_k\) and \((K_3 \conv R_k)_k\) are strongly  precompact in \(\Lrm^1\), and that 
\[
[K_1 \conv Z_k]_{1,\infty}:=\sup_{\lambda>0} \lambda \, \absb{ \setb{x}{ |(K_1 \conv Z_k)(x)|>\lambda }} \lesssim |Z_k|(B_1)\to 0.
\]
Furthermore, one easily checks that \(K_1 \conv Z_k \to 0\) in the distributional sense. It is then straightforward to combine the above facts with the \emph{positivity} of \(\sigma_k\) (see~\cite[Lemma~2.2]{DePhilippisRindler16} for details) to deduce that also the sequence \((\varphi\sigma_k)_k\) is precompact in \(\Lrm^1\), whereby
\[
\abs{ \varphi \sigma_k-\varphi \sigma}(B_1)\to 0.
\]
This is, however, in contradiction with  \(\sigma\) being  absolutely continuous (which follows from Proposition~\ref{prop:P_neq_a_odot_b_2D}) and  \(\sigma_k\) being singular. Indeed, if we let \(G_k\) be the null set where \(\sigma_k\) is concentrated, we obtain that 
\[
0<\abs{\sigma}(B_{1/2})=\abs{\sigma}(B_{1/2}\setminus G_k)=\abs{\sigma-\sigma_k}(B_{1/2}\setminus G_k)\le \abs{ \varphi (\sigma_k- \sigma)}(B_1)\to 0,
\]
which is impossible. \qed

\subsection{Local structure of singularities}\label{s:singloc}

As we mentioned in the introduction, Alberti's rank-one theorem,  Theorem~\ref{thm:Bvstruct}~(iii), implies a strong constraint on the possible behaviors of singularities of BV-maps. In particular, even at points $x_0 \in \Omega$ around which $u \in \BV(\Omega;\R^m)$ has a Cantor-type (e.g.\ fractal) structure, the \enquote{slope} of $u$ has a well-defined \emph{direction}. This is made precise in the following important consequence of Alberti's theorem.

\begin{corollary} \label{cor:Alberti_1D}
Let $u \in \BV_\loc(\R^d;\R^\ell)$. Then, at $\abs{D^s u}$-almost every $x_0$ every tangent measure $\sigma \in \Tan(D^s u,x_0)$ is $b$-directional for some direction $b \in \Sbb^{d-1}$ in the sense that
\[
  \sigma(B + v) = \sigma (B)
\]
for all bounded Borel sets $B \subset \R^d$ and all $v \in \R^d$ orthogonal to $b$.
\end{corollary}

For the proof see for instance~\cite[Corollary~10.8]{Rindler18book}.

Combining Theorem~\ref{thm:AlbertiBD}  with Theorem~\ref{thm:general rigidity} one can obtain some structural information on tangent measures for \(\BD\) maps. In fact, also exploiting the decomposition~\eqref{eq:Tan_polar}, which involves only \emph{positive} measures after the fixed polar, the structure results of Theorem~\ref{thm:general rigidity} can be improved for tangent measures.\footnote{We gratefully acknowledge Adolfo Arroyo-Rabasa for pointing this out to us.}

\begin{theorem} \label{thm:BD_tangent_decomp}
Let $u \in \BD_\loc(\R^d)$. Then, at  all point such that $\abs{E^s u}$-almost every $x_0$  the following holds: for all \(\sigma \in \Tan(|E^su|,x_0) \) there exists \(w\in \BD_\loc(\R^d)\) such that
\[
Ew = (a \odot b) \sigma,
\]
where $a,b \in \R^d$ are such that 
\[
\frac{\di E^s u}{\di |E^s u|}(x_0)=a\odot b.
\]
 Moreover:
 \begin{itemize}
 \item[(i)] If $a \neq \pm b$, then there exist two functions \(H_1,H_2\in \BV_\loc(\R)\) such that 
\[
 w(x)= aH_1(x\cdot b)+b H_2(x\cdot a).
\]
 \item[(ii)] If $a = \pm b$, then there exist a function \(H\in \BV_\loc(\R)\) such that 
 \[
w(x)= aH(x\cdot a).
 \]
 \end{itemize}
\end{theorem}

\begin{proof}
Given a tangent measure \(\sigma \in \Tan(|E^su|,x_0) \), by arguing as in the proof of Theorem~\ref{thm:AlbertiBD}, one gets a sequence \(r_k\todown 0\) and a sequence of rigid deformations \(\omega_k\) such that the maps 
\[
v_k(y):= \frac{r_k^{d-1}}{|E^s u|(B_{r_k}(x_0))} \, u(x_0+r_k y)+\omega_k(y), \qquad y \in \R^d,
\]
converge to a map \(w \in \BD_\loc(\R^d)\) with
\[
Ew=\frac{\di Eu}{\di\abs{Eu}}(x_0) \, \sigma.
\]
By Theorem~\ref{thm:AlbertiBD}, 
\[
\frac{\di Eu}{\di\abs{Eu}}(x_0)=a\odot b
\]
for some $a,b \in \R^d$. Thus, case~(i) or case~(ii) of Theorem~\ref{thm:general rigidity} applies. 
Assume for instance \(a\ne \pm b\). Then, 
\[
\begin{split}
 w(x)&=a\bigl(H_1(x\cdot b)+(x\cdot b)(x\cdot v)\bigr)+b \bigl(H_2(x\cdot a)+(x\cdot a)(x\cdot v)\bigr)
 \\
 &\qquad - v \, (x\cdot a)(x\cdot b)+\omega(x)
\end{split}
\]
and, by \eqref{e:adolfo},
\[
	\sigma=H_1'(\di x \cdot b)+H_2'(\di x \cdot a)+2(x \cdot v) \, \Lcal^{d}(\di x).
\]
First, we observe that we may assume $\omega = 0$ since we may just subtract it from $w$.

We claim that since   \(\sigma\) is a positive measure and  \(v \in \spn\{a,b\}^\perp\), this implies that \(v=0\), so that the conclusion holds. To prove the claim, assume without loss of generality that \(a = \ee_3\), \(b=\ee_2\) and that \(v=\alpha \ee_1\) with \(\alpha \ge 0\).  Let \(\varphi\in \Crm^0_c(\R^{d-1};[0,1])\), \(\psi \in  \Crm^0_c([0,1];[0,1])\) and let \(t\in \R\). By integrating \(\sigma\) against \(\varphi(x')\psi (x_1-t)\) (\(x=(x_1,x')\)) we get
\begin{align*}
0 &\leq \int \varphi(x')\psi (x_1-t) \dd \sigma\\
&\leq \biggl(\int \psi \dd \Lcal^1\biggr) \cdot \biggl(\int \varphi \dd H_1'(x_2)\dd \Lcal^{d-1}+\int \varphi \dd H_2'(x_3)\dd \Lcal^{d-1}\biggr)
\\
&\qquad + 2 \alpha  \int \varphi  \dd \Lcal^{d-1} \cdot \int_{t}^{t+1}y \dd \Lcal^{1}(y).
\end{align*}
Since the first term on the right hand side of the above equation  is independent of \(t\),  by letting \(t\to -\infty\) we get that \(\alpha=0\), which is the desired conclusion.  In the same way, if \(a=\pm b\), one uses Theorem~\ref{thm:general rigidity} (ii),  \eqref{e:adolfo2}, and the positivity of \(\sigma\) to conclude in a similar way.
\end{proof}

Note that according to the preceding result the structure of possible  tangent BD-maps can be quite complicated. However, if we additionally know \(x_0 \in J_u\), as a consequence of the structural Theorem~\ref{thm:BdstructIntro} we obtain that the tangent map at this point has a much simpler structure, namely
\[
w=w^+\ONE_{\{x\cdot n >0\}}+w^-\ONE_{\{x\cdot n <0\}}
\]
for some $w^\pm \in \R^d$ and $n \in \Sbb^{d-1}$ (in fact, $n = \nu_{J_u}$); in particular, $w$ is one-directional.

At a generic point we can still prove that there is always \emph{at least one} one-directional tangent measure. Indeed, one has the following result, proved in~\cite[Lemma~2.14]{DePhilippisRindler17}:

\begin{theorem}[Very good singular blow-ups] \label{thm:very_good_blowups}
Let $u \in \BD_\loc(\R^d)$. Then, at $\abs{E^s u}$-almost every $x_0$  \emph{there exist} \(\sigma \in \Tan(|E^su|,x_0)\) and  \(w\in \BD_\loc(\R^d)\) such that
\[
Ew =(a\odot b) \sigma
\]
where $a,b \in \R^d$ are such that 
\[
\frac{\di E^s u}{\di |E^s u|}(x_0)=a\odot b,
\]
and 
\[
  w(x) =  \eta G(x \cdot \xi) + A(x).
\]
Here, \(\{\xi, \eta\}=\{a,b\}\),   $G \in \BV_\loc(\R)$, and $A \colon \R^d \to \R^d$ is an affine map.
\end{theorem}

\begin{proof}[Sketch of the proof]
The idea of the proof is to start with  a  tangent map \(w\) as in Theorem~\ref{thm:BD_tangent_decomp} and to take a further blow-up in order to end up in the above situation and to appeal to a theorem of Preiss  that \emph{tangent measures to tangent measures are tangent measures}, see~\cite[Theorem~14.16]{Mattila95book}. One needs to distinguish two  cases:

In the case where \(H_1'(x\cdot b)\) and \(H_2'(x\cdot a)\) do not have singular parts, one simply takes a Lebesgue point of both of them and blows up around that point. Thus one finds an affine tangent map.
In the case where \(H_1'(x\cdot b)\) has a singular part, one easily checks that \(D^s H_1(x\cdot a)\) is  singular with respect to  \(H_2'(x\cdot a)\) and hence, taking a suitable blow-up, one again ends up with a \(w\) of the desired form.

We refer to~\cite[Lemma~2.14]{DePhilippisRindler17} for the details.
\end{proof}

Note that in the above theorem one cannot decide \emph{a-priori} which of the two directions \(a,b\) will appear  in the second blow-up. Furthermore, it can happen that the roles of \(a\) and \(b\) differ depending on the blow-up sequence.  In view of the analogy with the rectifiable part, where only one-directional measures are seen as possible tangent measures, one might formulate the following conjecture:

\begin{conjecture}\label{conj:strcuture}
For \(\abs{E^su}\)-almost all \(x_0\), the conclusion of Theorem~\ref{thm:very_good_blowups} holds for \emph{every} tangent measure \(\sigma\in \Tan(E^su,x_0)\).
\end{conjecture}

Note that if verified, this statement would imply that the structure of the Cantor part (which can be thought of as containing ``infinitesimal'' discontinuities) is essentially the same as the jump part (which contains macroscopic discontinuities).

\subsection{Dimensionality and rectifiability} \label{ssc:dim}

In~\cite{ArroyoRabasaDePhilippisHirschRindler19} it was shown that the approach used to prove Theorem~\ref{thm:Afree} can be extended to recover some information about dimensionality and rectifiability of \(\Acal\)-free measures. Indeed, it turns out that if an \(\Acal\)-free measure \(\mu\) charges a ``low-dimensional'' set, then its polar vector \(\frac{\di\mu}{\di |\mu|}\) has to satisfy a strong constraint at $\abs{\mu}$-almost every point in this set. To state this properly, let us introduce the following family of cones:
\[
	\Lambda^h_\Acal:=\bigcap_{\pi \in \Gr(h,d)} \bigcup_{\substack{\xi \in \pi\setminus\{0\} }} \ker\mathbb A(\xi),  \qquad h = 1, \ldots, d,
\]
where \(\mathbb A(\xi)\) is defined in~\eqref{e:symbol} and  \(\Gr(h,d)\) is the Grassmannian of \(h\)-planes in \(\R^d\). Note that
\begin{equation*}
\Lambda^{1}_{\Acal}=\bigcap_{\xi\in \R^d\setminus \{0\}} \ker \mathbb A(\xi)\subset\Lambda_{\Acal}^j \subset\Lambda_{\Acal}^h\subset \Lambda_{\Acal}^d=\Lambda_{\Acal},\qquad 1\le j\le h\le d. 
\end{equation*}
We also recall the definition of the \(h\)-dimensional \emph{integral geometric measure}, see~\cite[Section~5.14]{Mattila95book},
\[
\Ical^h(E):=\int_{\mathrm{Gr}(h,d)} \int_{\pi}\Hcal^{0}(E\cap \mathrm{proj}_{\pi}^{-1}(x)) \dd \Hcal^h(x) \dd \gamma_{h,d}(\pi),
\]
where \(\gamma_{h,d}\) is the Haar measure on the Grassmannian. The main result of~\cite{ArroyoRabasaDePhilippisHirschRindler19} is the following, see~\cite[Theorem 1.3]{ArroyoRabasaDePhilippisHirschRindler19}:

\begin{theorem}[Dimensional restrictions on polar] \label{t:dimens}
	Let $\mu \in \Mcal(\Omega;\R^m)$ be $\Acal$-free and let  \(E \subset \R^d\) be  a Borel set with \(\mathcal I^{h}(E) = 0\) for some \(h\in \{1,\ldots,d\}\). Then,
\[
\frac{\di \mu}{\di |\mu|}(x) \in \Lambda^h_\Acal\qquad\mbox{for \(|\mu|\)-a.e.\ \(x\in E\)}.
\]
\end{theorem}

Note that for \(h=d\) this theorem coincides with Theorem~\ref{thm:Afree}. The following is a straightforward corollary, see~\cite[Corollary 1.4]{ArroyoRabasaDePhilippisHirschRindler19}:

\begin{corollary}[Dimensionality]\label{cor:aac}
Let \(\Acal\) and \(\mu\) be as in Theorem~\ref{t:dimens} and assume that \(\Lambda^h_{\Acal}=\{0\}\) for some \(h \in \{1,\ldots,d\}\).
Then,
\[
   \text{$E \subset \R^d$ Borel with $\Ical^{h}(E)=0$}  \quad\Longrightarrow\quad
   |\mu|(E)=0.
\]
In particular,
\[
  \mu \ll \Ical^h\ll\Hcal^h
\]
and thus 
\begin{equation*}\label{e:dim}
\dim_{\Hcal} \mu:=\sup \, \setb{ h > 0 }{ \mu\ll\Hcal^h} \ge h_\Acal,
\end{equation*}
where
\begin{equation*}\label{ell1}
  h_\Acal : =\max\setb{ h \in \{1,\ldots,d\}} { \Lambda_{\Acal}^h=\{0\}}.
\end{equation*}
\end{corollary}

By combining the above corollary with the Besicovitch--Federer rectifiability criterion, see~\cite[Section 3.3.13]{Federer69book}, one obtains that  for an \(\Acal\)-free measure its $h$-dimensional parts are rectifiable whenever \(\Lambda^{h}_{\Acal}=\{0\}\). Recall that for a positive measure \(\sigma\) its \term{\(h\)-dimensional upper density} at a point $x$ is defined as 
	\[
		\theta^*_h(\sigma)(x) := \limsup_{r \to 0} \frac{\sigma(B_r(x))}{(2r)^h}.
	\]  
We then have, see~\cite[Theorem 1.5]{ArroyoRabasaDePhilippisHirschRindler19}:

\begin{theorem}[Rectifiability]\label{t:rect}
Let \(\Acal\) and \(\mu\) be as in Theorem~\ref{t:dimens} and assume that \(\Lambda^{h}_{\Acal}=\{0\}\). Then, the set $\{\theta^*_h(|\mu|) = +\infty\}$ is $|\mu|$-negligible and \(\mu\restrict \{\theta^{*}_{h}(|\mu|)>0\}\) is concentrated on an \(h\)-rectifiable set \(R\), that is,
\[
\mu \restrict \{\theta^{*}_{h}(|\mu|)>0\}=\lambda\,\Hcal^h\restrict R,
\]
where \(\lambda \colon R \to \mathbb \R^{m}\) is $\Hcal^h$-measurable.
\end{theorem}

The above results also imply a new proof of the rectifiability of the \((d-1)\)-dimensional part of derivatives of BV-maps and of symmetrized derivatives of BD-maps. Indeed, it suffices to notice that, by direct computations,
\[
\Lambda_{\curl}^{d-1}=\{0\}, \qquad \Lambda_{\Curl\Curl}^{d-1}=\{0\}.
\]
This recovers item~(ii) in Theorems~\ref{thm:Bvstruct} and~\ref{thm:BdstructIntro}. We refer the reader to~\cite{ArroyoRabasaDePhilippisHirschRindler19} for a more detailed discussion.

\section{Integral functionals and Young measures}\label{sec:app}

In this section we consider integral functionals of the form
\begin{equation} \label{eq:FLD}
  \Fcal[u] := \int_\Omega f(x,\Ecal u(x)) \dd x,  \qquad u \in \LD(\Omega),
\end{equation}
where $\Omega$ is a bounded Lipschitz domain, $f \colon \Omega \times \Rdds \to [0,\infty)$ is a Carath\'{e}odory integrand (Lebesgue measurable in the first argument and continuous in the second argument) with \term{linear growth at infinity}, that is,
\[
  f(x,A) \leq C(1+\abs{A})  \qquad\text{for some $C>0$ and all $A \in \Rdds$,}
\]
and the subspace $\LD(\Omega)$ of $\BD(\Omega)$ consists of all BD-maps such that $Eu$ is absolutely continuous with respect to Lebesgue measure (i.e., $E^s u = 0$).

Recall that a sequence \((u_j)\) is said to  \term{weak*-converge} to \(u\) in \(\BD(\Omega)\),  in symbols $u_j \toweakstar u$, if  $u_j \to u$ strongly in $\Lrm^1(\Omega;\R^d)$ and $Eu_j \toweakstar Eu$ in $\Mcal(\Omega;\R_{\rm sym}^{d\times d})$. Moreover, $(u_j)$ converges \term{strictly} or \term{area-strictly} to $u$ if $u_j \toweakstar u$ in $\BD(\Omega)$ and additionally $\abs{Eu_j}(\Omega) \to \abs{Eu}(\Omega)$ or $\langle Eu_j \rangle(\Omega) \to \langle Eu \rangle(\Omega)$, respectively. Here, for $u \in \BD(\Omega)$, we define the \term{(reduced) area functional} $\langle Eu \rangle(\Omega)$ as
\[
  \langle Eu \rangle(\Omega) := \int_\Omega \sqrt{1+\abs{\Ecal u(x)}^2} \dd x + \abs{E^s u}(\Omega).
\]
Since $\LD(\Omega)$ is area-strictly dense in $\BD(\Omega)$ (by a mollification argument, see, e.g., Lemma~11.1 in~\cite{Rindler18book} for the corresponding argument for the density of $\Wrm^{1,1}(\Omega)$ in the space $\BV(\Omega)$), one can show the following result, whose proof is completely analogous to the BV-case; see, for instance, Theorem~11.2 in~\cite{Rindler18book}.

\begin{proposition} \label{prop:F_strictly_cont_ext_BD}
Let $f \colon \cl{\Omega} \times \Rdds \to [0,\infty)$ be continuous and such that the  \term{(strong) recession function}
\begin{equation} \label{eq:f_infty}
  f^\infty(x,A) := \lim_{\substack{\!\!\!\! x' \to x \\ \!\!\!\! A' \to A \\ \; t \to \infty}}
    \frac{f(x',tA')}{t},  \qquad x \in \cl{\Omega}, \, A \in \Rdds,
\end{equation}
exists. Then, the area-strictly continuous extension of the functional $\Fcal$ defined in~\eqref{eq:FLD} onto the space $\BD(\Omega)$ is
\begin{align*}
  \overline{\Fcal}[u] := \int_\Omega f \bigl( x, \Ecal u(x) \bigr) \dd x
  + \int_\Omega f^\infty \biggl( x, \frac{\di E^s u}{\di |E^s u|}(x) \biggr) \dd |E^s u|(x),  \qquad u \in \BD(\Omega).
\end{align*}
\end{proposition}

Note that, clearly, $f^\infty$ is positively $1$-homogeneous in $A$, that is $f^\infty(x,\alpha A) = \alpha f^\infty(x,A)$ for all $\alpha \geq 0$. Moreover, the existence of $f^\infty$ entails that $f$ has linear growth at infinity.

While the above result gives a way to extend $\Fcal$ to all of $\BD(\Omega)$, at least for some integrands, in general neither $\mathcal{F}$ nor $\overline{\Fcal}$ admit a minimizer. Usually, this occurs if $\overline{\Fcal}$ is not weakly* lower semicontinuous. In this situation we define the \term{relaxation} $\Fcal_{*}$ of $\mathcal{F}$ onto $\BD(\Omega)$ as
\[
\Fcal_{*}[u,\Omega]:=\setBB{\liminf_{j\to\infty} \Fcal[u_j,\Omega] }{ \text{$(u_j)\subset\LD(\Omega)$, $u_j \toweakstar u$ in $\BD(\Omega)$} }.
\]
Our first task is to identify $\Fcal_{*}$ as an integral functional, which will entail a suitable (generalized) convexification of the integrand.

\subsection{Symmetric-quasiconvexity}

The appropriate generalized convexity notion related to symmetrized gradients is the following: We call a bounded Borel function $f \colon \R_\sym^{d \times d} \to \R$ \term{symmetric-quasiconvex} if
\[
  f(A) \leq \dashint_D f(A + \Ecal \psi(y)) \dd y
  \qquad \text{for all $\psi \in \Wrm^{1,\infty}_0(D;\R^d)$ and all $A \in \R_\sym^{d \times d}$,}
\]
where $D \subset \R^d$ is any bounded Lipschitz domain (the definition is independent of the choice of $D$ by a covering argument). Similar assertions to the ones for quasiconvex functions hold, cf.~\cite{Ebobisse00,BarrosoFonsecaToader00}. In particular, if $f$ has linear growth at infinity, we may replace the space $\Wrm_0^{1,\infty}(D;\R^d)$ in the above formula by $\Wrm_0^{1,1}(D;\R^d)$ or $\LD_0(D)$ ($\LD$-functions with zero boundary values in the sense of trace), see Remark~3.2 in~\cite{BarrosoFonsecaToader00}.

Using one-directional oscillations, one can prove that if the function $f \colon \R_\sym^{d \times d} \to \R$ is symmetric-quasiconvex, then it holds that
\begin{equation} \label{eq:sym_qc_aodotb_conv}
  f( \theta A + (1-\theta) B ) \leq \theta f(A) + (1-\theta) f(B)
\end{equation}
whenever $A,B \in \R_\sym^{d \times d}$ with $B - A = a \odot b$ for some $a,b \in \R^d$ and $\theta \in [0,1]$, cf.\ Proposition~3.4 in~\cite{FonsecaMuller99} for a more general statement in the framework of $\Acal$-quasiconvexity.

If we consider $\R_\sym^{d \times d}$ to be identified with $\R^{d(d+1)/2}$ and $f \colon \R_\sym^{d \times d} \to \R$ with $\tilde{f} \colon \R^{d(d+1)/2} \to \R$, then the convexity in~\eqref{eq:sym_qc_aodotb_conv} implies that $\tilde{f}$ is separately convex (i.e., convex in every entry separately) and so, $f$ is locally Lipschitz continuous, see for example Lemma~2.2 in~\cite{BallKirchheimKristensen00}. If additionally $f$ has linear growth at infinity, then \emph{loc.~cit.} even implies that $f$ is globally Lipschitz continuous. 

Notice that from Fatou's lemma we get that the recession function $f^\infty$, if it exists, is symmetric-quasiconvex whenever $f$ is; this is completely analogous to the situation for ordinary quasiconvexity. Hence, $f^\infty$ is also continuous on $\R_\sym^{d \times d}$ in this situation.

We mention that non-convex symmetric-quasiconvex functions with linear growth at infinity exist. One way to construct such a function (abstractly) is the following: We define the \term{symmetric-quasiconvex envelope} $SQf \colon \Rdds \to \R \cup \{-\infty\}$ of a locally bounded Borel-function $f \colon \Rdds \to \R$ as
\begin{equation} \label{eq:SQf}
  SQf(A) := \inf \, \setBB{ \dashint_{B_1} f(A + \Ecal \psi(z)) \dd z }
  { \psi \in \Wrm_0^{1,\infty}(B_1;\R^d) },
\end{equation}
where $A \in \Rdds$. Clearly, $SQf \leq f$. Furthermore, if $f$ has $p$-growth, we may replace the space $\Wrm_0^{1,\infty}(B_1;\R^d)$ by $\Wrm_0^{1,p}(B_1;\R^d)$ via a density argument.

Just as for the classical quasi-convexity  one can show the following, cf.~\cite[Proposition~3.4]{FonsecaMuller99}:

\begin{lemma} \label{lem:Qh_QC}
For a continuous function $f \colon \Rdds \to [0,\infty)$ with $p$-growth, $p \in [1,\infty)$, the symmetric-quasiconvex envelope $SQf$ is symmetric-quasiconvex.
\end{lemma}

We then have the following class of symmetric-quasiconvex, but not convex, functions:

\begin{lemma} \label{lem:sym_qc_construction}
Let $F \in \R_\sym^{d \times d}$ be a matrix that cannot be written in the form $a \odot b$ for any $a,b \in \R^d$  and let $p \in [1,\infty)$. Define
\[
  h(A) := \dist(A,\{-F,F\})^p,  \qquad A \in \R_\sym^{d \times d}.
\]
Then, $SQh(0) > 0$ and the symmetric-quasiconvex envelope $SQh$ is not convex (at zero).
\end{lemma}

We  sketch here the proof  since it is quite illuminating and shows a connection to the wave cone of the $\Curl\Curl$-operator. 

\begin{proof}[Proof of Lemma~\ref{lem:sym_qc_construction}]
The key point is to show that  $SQh(0) > 0$. Then, if \(SQh\) were convex,
\[
  SQh(0) \leq \frac{1}{2} \bigl(SQh(-F) + SQh(F) \bigr)
  \leq \frac{1}{2} \bigl(h(-F) + h(F) \bigr)
  = 0,
\]
a contradiction.

To prove that \(SQh(0) > 0\)  we argue by contradiction and assume the existence of  a sequence of maps $(\psi_j) \subset \Wrm_0^{1,\infty}(B_1;\R^d)$ such that 
\begin{equation} \label{eq:qc_constr_0}
  \dashint_{B_1} h(\Ecal \psi_j) \dd z \to 0.
\end{equation}
In particular,
\begin{equation}\label{e:distconv}
\dist(\Ecal \psi_j,\{-F,F\})\to 0 \qquad \text{in \(\Lrm^p(B_1)\).}
\end{equation}
By  mollification, we can assume that the \(\psi_j\) are smooth and we extend them by zero to \(\R^d\). This allows one to   employ the Fourier transform like in the proof of Theorem~\ref{thm:general rigidity}~(iii). Recall that  for a symmetrized gradient one has
\[
\widehat{\Ecal w}(\xi) =(2\pi \ii) \,\hat{w}(\xi)\odot \xi,\qquad \xi \in \R^d.
\]
We let \(L:=\spn{F}\) and we denote by \(  \Pbf\) the orthogonal projection onto \(L^\perp\)  (identified with its complexification). We have already seen in~\eqref{eq:qc_multiplier} that we may \enquote{invert} $\Pbf $ in the sense that if
\[
  \Pbf(\widehat{\Ecal w}) = \hat{R}
\]
for some $w \in \Wrm^{1,p}(\R^d;\R^m)$, $R \in \Lrm^p(\R^d;L^\perp)$, then
\[
  \widehat{\Ecal w}(\xi)
  = \Mbf(\xi) \hat{R}(\xi)
  = \Mbf(\xi) \Pbf(\widehat{\Ecal w}(\xi)),  \qquad
  \xi \in \R^d \setminus \{0\},
\]
for some family of linear operators $\Mbf(\xi) \colon \Rdds \to \Rdds$ that depend smoothly and positively \mbox{$0$-homogeneously} on $\xi$. Then we conclude as follows:

For $p = 2$, Plancherel's identity $\norm{g}_{\Lrm^2} = \norm{\hat{g}}_{\Lrm^2}$ together with~\eqref{e:distconv} implies
\begin{align*}
  \norm{\Ecal \psi_j}_{\Lrm^2}
  &= \norm{\widehat{\Ecal \psi_j}}_{\Lrm^2} \\
  &= \norm{\Mbf(\xi) \Pbf(\widehat{\Ecal \psi_j}(\xi))}_{\Lrm^2} \\
  &\leq \norm{\Mbf}_\infty \norm{\Pbf(\widehat{\Ecal \psi_j}(\xi))}_{\Lrm^2} \\
  &= \norm{\Mbf}_\infty \norm{\Pbf(\Ecal \psi_j)}_{\Lrm^2} \\
  &\to 0.
\end{align*}
But then $h(\Ecal \psi_j) \to \abs{F}$ in $\Lrm^1(B_1)$, contradicting~\eqref{eq:qc_constr_0}. Thus, $SQh(0) > 0$.

For $p \in (1,\infty)$, we may apply the Mihlin multiplier theorem (see for instance~\cite[Theorem~5.2.7]{Grafakos14book1}) to get analogously that
\[
  \norm{\Ecal \psi_j}_{\Lrm^p} \leq C \norm{\Mbf}_{\Crm^{\floor{d/2} + 1}} \norm{\Pbf(\Ecal \psi_j)}_{\Lrm^p} \to 0,
\]
which is again at odds with~\eqref{eq:qc_constr_0}.

For $p = 1$, we only have the weak-type estimate
\[
[\Ecal \psi_j]_{1,\infty}\le  C \norm{\Mbf}_{\Crm^{\floor{d/2} + 1}} \norm{\Pbf(\Ecal \psi_j)}_{\Lrm^1}\to 0,
\]
see~\cite[Theorem~5.2.7]{Grafakos14book1}. This in particular implies that, up to a subsequence, it holds that $\Ecal \psi_j \to 0$ almost everywhere. On the other hand, by the trivial estimate
\[
\abs{\Ecal \psi_j(x)}\le C\bigl(1+\dist(\Ecal \psi_j(x),\{-F,F\})\bigr)
\]
in conjunction with~\eqref{e:distconv} we deduce that  \(\abs{\Ecal \psi_j}\) is  equiintegrable. By Vitali's theorem, \(\Ecal \psi_j\to 0\) in \(\Lrm^1 \) and we conclude as above.
\end{proof}

An alternative way to show that \(SQh(0)>0\) would be to rely on the  following generalization of the Ball--James theorem~\cite{BallJames87} on approximate rigidity for the two-state problem. 

\begin{theorem}\label{thm:twostate}
Let \(p \in [1,\infty)\) and  \(A, B\in \R_\sym^{2 \times 2} \) be such that \(B-A\ne a \odot b\) for any $a,b \in \R^2$. Let  $(\psi_j) \subset \Wrm_0^{1,\infty}(B_1;\R^d)$ be a sequence of maps such that 
\[
\dist(\Ecal \psi_j,\{A,B\})\to 0 \qquad \text{in \(\Lrm^p(B_1)\).}
\]
Then, up to a subsequence, either \(\Ecal \psi_j\to A\)  or  \(\Ecal \psi_j\to B\)  in \(\Lrm^p\).
\end{theorem}

Indeed, applying the above lemma to \(\{A,B\}=\{-F,F\}\), one  obtains  that either \(\Ecal \psi_j \to F\)  or \(\Ecal \psi_j \to -F\) in \(\Lrm^p\), in contradiction with  the fact that 
\[
\int_{B_1}\Ecal \psi_j=0.
\]
 which  follows from the zero trace assumption on \(\psi_j\in \Wrm_0^{1,\infty}(B_1;\R^d)\).

Theorem~\ref{thm:twostate}  can in fact be proved in  the more general context of \(\Acal\)-free measures;  we refer the reader to~\cite{DePhilippisPalmieriRindler18} (which is based on the techniques of~\cite{DePhilippisRindler16}). For the case of first-order operators $\Acal$, this result is also proved in~\cite{ChiodaroliFeireislKremlWiedemann17}.

Let us also remark that recently there has been a detailed investigation into symmetric polyconvexity, see~\cite{BoussaidKreisbeckSchlomerkemper19}.

\subsection{Relaxation}

We now consider the question raised at the beginning of this section, namely to identify the relaxation $\Fcal_*$ of $\Fcal$. First results in this direction for functions in $\BD(\Omega)$, but without a Cantor part (i.e., the singular part $E^s u$ originates from jumps only and does not contain Cantor-type measures), were proved in~\cite{BellettiniCosciaDalMaso98,BarrosoFonsecaToader00,Ebobisse05,GargiuloZappale08}.

The first lower semicontinuity theorem applicable to the whole space $\BD(\Omega)$ was proved in~\cite{Rindler11} by employing the results of Section~\ref{sec:rigidity} together with a careful analysis of tangent measures and (iterated) tangent Young measures (see Section~\ref{ssc:tanYM} below). That work, however, left open the question of relaxation, where more information on the structure of the singular part is required. In this context, we refer to~\cite{AmbrosioDalMaso92,FonsecaMuller93}, where this question is treated for BV-maps via Alberti's rank-one theorem (and Corollary \ref{cor:Alberti_1D}), and to~\cite{Rindler12}, which shows that Alberti's rank-one theorem is not necessary to prove weak* lower semicontinuity in BV (without a full relaxation theorem). In \(\BD\), a first intermediate relaxation result was obtained in~\cite{ArroyoRabasaDePhilippisRindler17?} and an essentially optimal version  was finally proved in~\cite{KosibaRindler19?}, see~Theorem~\ref{thm:lsc-bd} below.

A challenge in the formulation of a relaxation theorem is that it involves passing to the symmetric-quasiconvex hull $SQf$ of the integrand (defined in~\eqref{eq:SQf}), but in general the (strong) recession function $(SQf)^\infty$ does not exist; in this context, we refer to~\cite[Theorem 2]{Muller92} for a counterexample. Thus, we need a more general notion of recession function: For any $f \in \Crm(\cl{\Omega} \times \Rdds)$ with linear growth at infinity we can always define the \term{generalized recession function} $f^\# \colon \cl{\Omega} \times \Rdds \to \R$ via
\[
  f^\#(x,A) := \limsup_{\substack{\!\!\!\! x' \to x \\ \!\!\!\! A' \to A \\ \; t \to \infty}} \,
  \frac{f(x',tA')}{t},  \qquad x \in \cl{\Omega}, \, A \in \Rdds,
\]
which again is always positively $1$-homogeneous and the linear growth at infinity of $f$ suffices for $f^\#$ to take only real values. In other works, $f^\#$ is usually just called the \enquote{recession function} (and denoted by ``$f^\infty$''), but here the distinction to between $f^\#$ and $f^\infty$ is important. It is elementary to prove that $f^\#(x,\frarg)$ is upper semicontinuous. For a convex function $f$, always $f^\# = f^\infty$. We refer to~\cite{AuslenderTeboulle03book} for a more systematic approach to recession functions and their associated cones.

As mentioned before, the following relaxation result is proved in~\cite{KosibaRindler19?}:

\begin{theorem}\label{thm:lsc-bd}
	Let $\Omega\subset\mathbb{R}^d$ be a bounded Lipschitz domain and let $f \colon \Rdds \to[0,\infty)$ be a continuous function such that there exist constants $0<c\leq C$, for which the inequality
	\[
	c|A|\leq f(A)\leq C(1+|A|), \qquad A\in\Rdds,
	\]
	holds.
	Then, the weak* relaxation of the functional $\Fcal$ in $\BD(\Omega)$ is given by
	\begin{align*}
	\Fcal_{*}[u]
	&=
	\int_{\Omega}(SQf)(\Ecal u(x)) \dd x
	+
	\int_{\Omega}(SQf)^{\#}\biggl(\frac{\di E^su}{\di |E^su|}(x)\biggr) \dd|E^su|(x), \quad u\in\BD(\Omega).
	\end{align*}
	In particular, $\Fcal_*$ is weakly* lower semicontinuous on $\BD(\Omega)$.
\end{theorem}

The proof of this theorem proceeds by the blow-up method, see e.g.~\cite{FonsecaMuller92,FonsecaMuller93}, and exploits Theorem~\ref{thm:AlbertiBD} as well as the existence of very good blow-ups from Theorem~\ref{thm:very_good_blowups}. 

We also note that~\cite{CarocciaFocardiVanGoethem19?} establishes a general integral representation theorem for the relaxed functional $\Fcal_*$.

We conclude this section by noting that the previous discussion applies to functionals  whose integrand depends  only on the symmetric part of the gradient.  However, it is not hard to construct an integrand \(f \colon \R^{d \times d}\to \R\),   which depends on the full  matrix, but for which 
\[
|A+A^T| \lesssim f(A)\lesssim 1+ |A+A^T|,  \qquad A \in \R^{d \times d}.
\] 
In this case the corresponding integral functional \(\Fcal\) will be coercive only on BD and one would like to study the relaxed functional \(\Fcal_{*}\). This has been achieved in some specific  cases when  \(E^c u=0\), see~\cite{ContiFocardiIurlano17, FriedrichSolombrino19}, but in the general case not much is known, see the discussion in~\cite[Section~7]{ContiFocardiIurlano17}.

\subsection{Generalized Young measures}

In the remainder of this survey we consider a more abstract approach to the theory of integral functionals defined on $\BD$, namely through the theory of generalized Young measures. These objects keep track of all oscillations and concentrations in a weakly* converging sequence of measures; here, we will apply this to the symmetrized derivatives $(Eu_j)$ of a weakly*-converging sequence of BD-maps $(u_j) \subset \BD(\Omega)$. Our presentation follows~\cite{AlibertBouchitte97,KristensenRindler10,Rindler11,DePhilippisRindler17,Rindler18book}, where also proofs and examples can be found.

Let, as usual, $\Omega \subset \R^d$ be a bounded Lipschitz domain. For $f \in \Crm(\cl{\Omega} \times \R^N)$ and $g \in \Crm(\cl{\Omega} \times \Bbb^N)$, where $\Bbb^N$ denotes the open unit ball in $\R^N$, we let
\begin{align*}
    (Rf)(x,\hat{A}) &:= (1-\absn{\hat{A}}) \, f \biggl(x, \frac{\hat{A}}{1-\absn{\hat{A}}} \biggr),
      \qquad x \in \cl{\Omega}, \, \hat{A} \in \Bbb^N. 
\end{align*}
Define
\[
  \Ebf(\Omega;\R^N) := \setb{ f \in \Crm(\cl{\Omega} \times \R^N) }{ \text{$Rf$ extends continuously onto $\cl{\Omega \times \Bbb^N}$} }.
\]
In particular, $f \in \Ebf(\Omega;\R^N)$ has linear growth at infinity with growth constant $C = \norm{Rf}_{\Lrm^\infty(\Omega \times \Bbb^N)}$. Furthermore, for all $f \in \Ebf(\Omega;\R^N)$, the \term{(strong) recession function} $f^\infty \colon \cl{\Omega} \times \R^N \to \R$, defined in~\eqref{eq:f_infty}, exists and takes finite values. It can be shown that in fact $f \in \Crm(\cl{\Omega};\R^N)$ is in the class $\Ebf(\Omega;\R^N)$ if and only if  $f^\infty$ exists in the sense~\eqref{eq:f_infty}.

A \term{(generalized) Young measure} $\nu \in \Ybf(\Omega;\R^N)$ on the open set $\Omega \subset \R^d$ with values in $\R^N$ is a triple $\nu = (\nu_x,\lambda_\nu,\nu_x^\infty)$ consisting of
\begin{itemize}
  \item[(i)] a parametrized family of probability measures $(\nu_x)_{x \in \Omega} \subset \Mcal_1(\R^N)$, called the \term{oscillation measure};
  \item[(ii)] a positive finite measure $\lambda_\nu \in \Mcal_+(\cl{\Omega})$, called the \term{concentration measure}; and
  \item[(iii)] a parametrized family of probability measures $(\nu_x^\infty)_{x \in \cl{\Omega}} \subset \Mcal_1(\Sbb^{N-1})$, called the \term{concentration-direction measure},
\end{itemize}
for which we require that
\begin{itemize}
  \item[(iv)] the map $x \mapsto \nu_x$ is weakly* measurable with respect to $\Lcal^d$, i.e., the function $x \mapsto \dpr{f(x,\frarg),\nu_x}$ is $\Lcal^d$-measurable for all bounded Borel functions $f \colon \Omega \times \R^N \to \R$;
  \item[(v)] the map $x \mapsto \nu_x^\infty$ is weakly* measurable with respect to $\lambda_\nu$; and
  \item[(vi)] $x \mapsto \dprn{\abs{\frarg},\nu_x} \in \Lrm^1(\Omega)$.
\end{itemize}

The \term{duality pairing} between $f \in \Ebf(\Omega;\R^N)$ and $\nu \in \Ybf(\Omega;\R^N)$ is given as
\begin{align*}
  \ddprb{f,\nu} &:= \int_\Omega \dprb{f(x,\frarg), \nu_x} \dd x
    + \int_{\cl{\Omega}} \dprb{f^\infty(x,\frarg),\nu_x^\infty} \dd \lambda_\nu(x) \\
  &:= \int_\Omega \int_{\R^N} f(x,A) \dd \nu_x(A) \dd x
  + \int_{\cl{\Omega}} \int_{\partial \Bbb^N} f^\infty(x,A) \dd \nu_x^\infty(A) \dd \lambda_\nu(x).
\end{align*}
The weak* convergence $\nu_j \toweakstar \nu$ in $\Ybf(\Omega;\R^N) \subset \Ebf(\Omega;\R^N)^*$ is then defined with respect to this duality pairing. If $(\gamma_j) \subset \Mcal(\cl{\Omega};\R^N)$ is a sequence of measures with $\sup_j \abs{\gamma_j}(\cl{\Omega}) < \infty$, then we say that the sequence $(\gamma_j)$ \term{generates} a Young measure $\nu \in \Ybf(\Omega;\R^N)$, in symbols $\gamma_j \toY \nu$, if for all $f \in \Ebf(\Omega;\R^N)$ it holds that
\begin{align*}
&f \biggl( x, \frac{\di \gamma_j}{\di \Lcal^d}(x)\biggr) \,\Lcal^d \restrict \Omega
+ f^\infty \biggl(x, \frac{\di \gamma_j}{\di \abs{\gamma_j}}(x) \biggr) \, \abs{\gamma^s_j}\\
&\qquad\toweakstar\;\; \dprb{f(x,\frarg), \nu_x} \, \Lcal^d \restrict \Omega + \dprb{f^\infty(x,\frarg),
\nu_x^\infty} \, \lambda_\nu  \qquad\text{in $\Mcal(\cl{\Omega})$.}
\end{align*}

Also, for $\nu \in \Ybf(\Omega;\R^N)$ we define the \term{barycenter} as the measure
\[
  [\nu] := \dprb{\id, \nu_x} \, \Lcal^d \restrict \Omega + \dprb{\id, \nu_x^\infty} \, \lambda_\nu
  \in \Mcal(\cl{\Omega};\R^N).
\]

The following is the central compactness result in $\Ybf(\Omega;\R^N)$:

\begin{lemma}[Compactness] \label{lem:compact}
Let $(\nu_j) \subset \Ybf(\Omega;\R^N)$ be such that
\[
  \supmod_j \, \ddprb{\ONE \otimes \abs{\frarg}, \nu_j} < \infty.
\]
Then, $(\nu_j)$ is weakly* sequentially relatively compact in $\Ybf(\Omega;\R^N)$, i.e., there exists a subsequence (not relabeled) such that $\nu_j \toweakstar \nu$ and $\nu \in \Ybf(\Omega;\R^N)$.
\end{lemma}

In particular, if $(\gamma_j) \subset \Mcal(\cl{\Omega};\R^N)$ is a sequence of measures with $\sup_j \abs{\gamma_j}(\cl{\Omega}) < \infty$ as above, then there exists a subsequence (not relabeled) and $\nu \in \Ybf(\Omega;\R^N)$ such that $\gamma_j \toY \nu$.

\subsection{BD-Young measures}

A Young measure in $\Ybf(\Omega;\R_\sym^{d \times d})$ is called a \term{BD-Young measure}, $\nu \in \BDY(\Omega)$, if it can be generated by a sequence of BD-symmetrized derivatives. That is, for all $\nu \in \BDY(\Omega)$ there exists a (necessarily norm-bounded) sequence $(u_j) \subset \BD(\Omega)$ with $Eu_j \toY \nu$. When working with $\BDY(\Omega)$, the appropriate space of integrands is $\Ebf(\Omega;\R_\sym^{d \times d})$ since it is clear that both $\nu_x$ and $\nu_x^\infty$ only take values in $\R_\sym^{d \times d}$ whenever $\nu \in \BDY(\Omega)$. It is easy to see that for a BD-Young measure $\nu \in \BDY(\Omega)$ there exists $u \in \BD(\Omega)$ satisfying $Eu = [\nu] \restrict \Omega$; any such $u$ is called an \term{underlying deformation} of $\nu$. 

The following results about BD-Young measures are proved in~\cite{KristensenRindler10,Rindler11,Rindler18book} (the references~\cite{KristensenRindler10,Rindler18book} treat BV-Young measures, but the proofs adapt line-by-line).

\begin{lemma}[Good generating sequences] \label{lem:good_genseq}
Let $\nu \in \BDY(\Omega)$.
\begin{itemize}
  \item[(i)] There exists a generating sequence $(u_j) \subset \BD(\Omega) \cap \Crm^\infty(\Omega;\R^d)$ with $Eu_j \toY \nu$.
  \item[(ii)] If additionally $\lambda_\nu(\partial \Omega) = 0$, then the $u_j$ from (i) can be chosen to satisfy $u_j|_{\partial \Omega} = u|_{\partial \Omega}$, where $u \in \BD(\Omega)$ is any underlying deformation of $\nu$.
\end{itemize}
\end{lemma}

The proof of this result can be found in~\cite[Lemma~4]{KristensenRindler10}.

\subsection{Tangent Young measures} \label{ssc:tanYM}

In order to carry out blow-up constructions involving Young measures, we will need localization principles for these objects, one at regular and one at singular points. These results should be considered complements to the theory of tangent measures and thus the Young measures obtained in the blow-up limit are called \term{tangent Young measures}.

Define $\BDY_\loc(\R^d)$ by replacing $\Ybf(\Omega;\R_\sym^{d \times d})$ and $\BD(\Omega)$ by their respective local counterparts. When working with $\BDY_\loc(\R^d)$, the appropriate space of integrands is $\Ebf_c(\R^d;\R_\sym^{d \times d})$, i.e., the set of all functions in $\Ebf(\R^d;\R_\sym^{d \times d})$ with (uniformly) compact support in the first argument.

The following two results are proved in~\cite{Rindler11}:

\begin{proposition}[Localization at regular points] \label{prop:localize_reg}
Let $\nu \in \BDY(\Omega)$ be a BD-Young measure. Then, for $\Lcal^d$-almost every $x_0 \in \Omega$ there exists a \term{regular tangent Young measure} $\sigma \in \BDY_\loc(\R^d)$ satisfying
\begin{align*}
  [\sigma] &\in \Tan([\nu],x_0),  &\sigma_y &= \nu_{x_0} \quad\text{a.e.,}   
  \\
  \lambda_\sigma &= \frac{\di \lambda_\nu}{\di \Lcal^d}(x_0) \, \Lcal^d  \in \Tan(\lambda_\nu,x_0),
    &\sigma_y^\infty &= \nu_{x_0}^\infty \quad\text{a.e.}
\end{align*}
In particular, for all bounded open sets $U \subset \R^d$ with $\Lcal^d(\partial U) = 0$, and all $h \in \Crm(\R^{d \times d})$ such that the recession function $h^\infty$ exists in the sense of~\eqref{eq:f_infty}, it holds that
\begin{equation*} \label{eq:loc_reg_ddpr}
  \ddprb{\ONE_U \otimes h, \sigma} = \biggl[ \dprb{h,\nu_{x_0}} + \dprb{h^\infty,\nu_{x_0}^\infty}
    \frac{\di \lambda_\nu}{\di \Lcal^d}(x_0) \biggr] \abs{U}.
\end{equation*}
\end{proposition}

\begin{proposition}[Localization at singular points] \label{prop:localize_sing}
Let $\nu \in \BDY(\Omega)$ be a BD-Young measure. Then, for $\lambda_\nu^s$-almost every $x_0 \in \Omega$, there exists a \term{singular tangent Young measure} $\sigma \in \BDY_\loc(\R^d)$ satisfying
\begin{align*}
  [\sigma] &\in \Tan([\nu],x_0),  &\sigma_y &= \delta_0 \quad\text{a.e.,} 
\\
  \lambda_\sigma &\in \Tan(\lambda_\nu^s,x_0) \setminus \{0\}, &\sigma_y^\infty &= \nu_{x_0}^\infty
    \quad\text{$\lambda_\sigma$-a.e.} 
\end{align*}
In particular, for all bounded open sets $U \subset \R^d$ with $(\Lcal^d + \lambda_\sigma)(\partial U) = 0$ and all positively $1$-homogeneous $g \in \Crm(\R_\sym^{d \times d})$ it holds that
\begin{equation*} \label{eq:loc_sing_ddpr}
  \ddprb{\ONE_U \otimes g, \sigma} = \dprb{g,\nu_{x_0}^\infty} \, \lambda_\sigma(U).
\end{equation*}
\end{proposition}

\subsection{Good blow-ups for Young measures}

By exploiting the results in Section~\ref{sec:singularities}, in particular Theorem~\ref{thm:very_good_blowups}, one can show that for  almost all singular points of a BD-Young measure there is a tangent Young measure such that  the underlying deformation has a one-directional structure,  see~\cite[Lemma~2.14]{DePhilippisRindler17}.

\begin{theorem}[Very good singular blow-ups] \label{thm:very_good_blowupsII}
Let $\nu \in \BDY(\Omega)$ be a BD-Young measure. Then, for $\lambda_\nu^s$-almost every $x_0 \in \Omega$, there exists a singular tangent Young measure $\sigma \in \BDY_\loc(\R^d)$ such that $[\sigma] = Ew$ for some $w \in \BD_\loc(\R^d)$ of the form
\[
  w(x) = \eta G(x \cdot \xi) + A(x).
\]
Here, $\xi, \eta \in \R^d \setminus \{0\}$,   $G \in \BV_\loc(\R)$, and $A \colon \R^d \to \R^d$ is an affine map.
\end{theorem}

We remark that the previous result also holds for (possibly non-BD) Young measures $\nu \in \Ybf(\Omega;\Rdds)$ with the property that $ [\nu] \restrict \Omega = Eu$ for some $u \in \BD(\Omega)$.

Using the previous theorem, the main result of~\cite{DePhilippisRindler17} characterizes completely all BD-Young measures:

\begin{theorem} \label{thm:BDY_charact}
Let $\nu \in \Ybf(\Omega;\Rdds)$ be a (generalized) Young measure. Then, $\nu$ is a BD-Young measure, $\nu \in \BDY(\Omega)$, if and only if there exists $u \in \BD(\Omega)$ with $[\nu] \restrict \Omega = Eu$ and for all symmetric-quasiconvex $h \in \Crm(\R_\sym^{d \times d})$ with linear growth at infinity, the Jensen-type inequality
\[
  h \biggl( \dprb{\id,\nu_x} + \dprb{\id,\nu_x^\infty} \frac{\di \lambda_\nu}{\di \Lcal^d}(x) \biggr)
    \leq \dprb{h,\nu_x} + \dprb{h^\#,\nu_x^\infty} \frac{\di \lambda_\nu}{\di \Lcal^d}(x).
\]
holds at $\Lcal^d$-almost every $x \in \Omega$.
\end{theorem}

This result is the generalization to BD of the so-called Kinderlehrer--Pedregal theorem characterizing classical Young measures (i.e., $\lambda_\nu = 0$) generated by sequences of gradients~\cite{KinderlehrerPedregal91,KinderlehrerPedregal94} and the characterization of generalized sequences generated by BV-derivatives, first established in~\cite{KristensenRindler10} and refined in~\cite{Rindler14,KirchheimKristensen11}.

We remark that the use of the generalized recession function $h^\#$ can in general not be avoided since, as discussed above,  not every symmetric-quasiconvex function with linear growth at infinity has a (strong) recession function (and one needs to test with all those; but see~\cite[Theorem~6.2]{KirchheimKristensen16} for a possible restriction on the class of test integrands). 

Note that the above theorem does not impose any constraint on the singular part (i.e., $\lambda_\nu^s$ and the corresponding $\nu_x^\infty$) of the Young measure $\nu$, except for the fact that the barycenter $[\nu]$'s polar is of the form \(a\odot b\) at almost every singular point (which follows by the existence of an underlying deformation and Theorem~\ref{thm:AlbertiBD}). It is a remarkable fact that this is enough to also ensure the validity of the following \emph{singular} Jensen-type inequality:

\begin{theorem} \label{thm:auto_Jensen}
For all $\nu \in \BDY(\Omega)$ and for all symmetric-quasiconvex $h \in \Crm(\R_\sym^{d \times d})$ with linear growth at infinity, it holds that
\begin{align*}
  h^\# \bigl( \dprb{\id,\nu_x^\infty} \bigr) &\leq \dprb{h^\#,\nu_x^\infty}
\end{align*}
at $\lambda_\nu^s$-almost every $x \in \Omega$.
\end{theorem}

The key step to proving the preceding theorem  is a surprising convexity property of  $1$-homogeneous  symmetric-quasiconvex functions at matrices of the form \(a \odot b\)  proved by  Kirchheim--Kristensen in~\cite{KirchheimKristensen16}:

\begin{theorem} \label{thm:KirchheimKristensen}
Let $h^\infty \colon \Rdds \to \R$ be positively $1$-homogeneous and symmetric-quasiconvex. Then, $h^\infty$ is convex at every matrix $a \odot b$ for $a, b \in \R^d$, that is, there exists an affine function $g \colon \Rdds \to \R$ with
\[
  h^\infty(a \odot b) = g(a \odot b)  \qquad\text{and}\qquad
  h^\infty \geq g.
\]
\end{theorem}

\begin{proof}[Proof of Theorem~\ref{thm:auto_Jensen}]
We first establish the following general claim: Let $\mu \in \Mcal_1(\Rdds)$ be a probability measure with barycenter $[\mu] := \dpr{\id,\mu} = a \odot b$ for some $a,b \in \R^d$, and let $h \in \Crm(\Rdds)$ be positively $1$-homogeneous and symmetric-quasiconvex. Then,
\[
  h(a \odot b) = h([\mu]) \leq \dpr{h,\mu}.
\]
Indeed, by the preceding theorem, $h$ is actually \emph{convex} at matrices $a \odot b$, that is, the classical Jensen inequality holds for measures with barycenter $a \odot b$, such as our $\mu$. This shows the claim.

If $\nu = (\nu_x,\lambda_\nu,\nu_x^\infty) \in \BDY(\Omega)$, then there is a $u \in \BD(\Omega)$ such that
\[
  Eu = [\nu_x] \, \Lcal^d_x + [\nu_x^\infty] \, (\lambda_\nu \restrict \Omega)(\di x),
\]
so at $\lambda_\nu^s$-almost every $x \in \Omega$ we have by Theorem~\ref{thm:AlbertiBD} that $[\nu_x^\infty] = a(x) \odot b(x)$ for some $a(x), b(x) \in \R^d$. Thus, applying the claim above to $\nu_x^\infty$ immediately yields the singular Jensen-type inequality.
\end{proof}

Together, Theorems~\ref{thm:BDY_charact} and~\ref{thm:auto_Jensen} have the following remarkable interpretation: While there are constraints on the oscillations and concentrations making up the absolutely continuous part of a BD-Young measure $\nu$, the concentrations in the singular part are totally unconstrained besides the requirement that $[\nu_x^\infty] \, (\lambda_\nu^s \restrict \Omega)(\di x) = E^s u$ for some $u \in \BD(\Omega)$. In particular, \emph{any} probability measure $\mu \in \Mcal(\Rdds)$ with barycenter $[\mu] = \dpr{\id,\mu} = a \odot b$ for some $a,b \in \R^d$ occurs as the concentration-direction measure of a BD-Young measure.


\providecommand{\bysame}{\leavevmode\hbox to3em{\hrulefill}\thinspace}
\providecommand{\MR}{\relax\ifhmode\unskip\space\fi MR }
\providecommand{\MRhref}[2]{%
  \href{http://www.ams.org/mathscinet-getitem?mr=#1}{#2}
}
\providecommand{\href}[2]{#2}

\end{document}